\documentclass[11pt]{amsart}

\usepackage{enumerate}
\usepackage[shortlabels]{enumitem}
\setlist{topsep=0pt, itemsep=0pt,parsep=0pt, leftmargin=*}

\usepackage{amsmath}
\usepackage{amssymb}
\usepackage{graphicx}
\usepackage{csquotes}
\usepackage{url}
\usepackage{amsmath,amssymb,amsthm}
\usepackage{mathtools}
\usepackage{xfrac}
\usepackage{sidecap} 
\usepackage[
   algo2e,
   lined,
   boxed,
  commentsnumbered,
  linesnumbered,
  ruled
 ]{algorithm2e}

\usepackage[bookmarks,
            raiselinks,
            pageanchor,
            hyperindex,
            colorlinks,
            citecolor=black,
            linkcolor=black,
            urlcolor=black,
            filecolor=black,
            menucolor=black]{hyperref}

\usepackage{pgfplots}
\pgfplotsset{compat=1.14}
\usepgfplotslibrary{groupplots}
\usepgfplotslibrary{dateplot}
\usepgfplotslibrary{units}
\usetikzlibrary{spy,backgrounds}
\usepackage{pgfplotstable}

\usepackage{subcaption} 
\usepackage{enumerate}
\usepackage{todonotes}
\usepackage{color}
\usepackage{setspace}
\definecolor{sixclassRdYlBu1}{rgb}{0.84,0.19,0.15}
\definecolor{sixclassRdYlBu2}{rgb}{0.99,0.55,0.35}
\definecolor{sixclassRdYlBu3}{rgb}{1.0,0.88,0.56}
\definecolor{sixclassRdYlBu4}{rgb}{0.88,0.95,0.97}
\definecolor{sixclassRdYlBu5}{rgb}{0.57,0.75,0.86}
\definecolor{sixclassRdYlBu6}{rgb}{0.27,0.46,0.71}

\DeclareRobustCommand{\rchi}{{\mathpalette\irchi\relax}}
\newcommand{\irchi}[2]{\raisebox{\depth}{$#1\chi$}}
\DeclarePairedDelimiterX\skal[2]{(}{)}{#1\,,\,#2}

\newtheorem{theorem}{Theorem}[section]
\newtheorem{corollary}[theorem]{Corollary}
\newtheorem{lemma}[theorem]{Lemma}
\newtheorem{proposition}[theorem]{Proposition}
\theoremstyle{definition}
\newtheorem{definition}[theorem]{Definition}

\newtheorem{remark}[theorem]{Remark}
\newtheorem*{propositionproof}{Proof of Proposition \ref{prop:hessian_Jhat_error_NCD}}
\newtheorem{assumption}{Assumption}

\numberwithin{equation}{section}

\newcommand{\hS}[1]{\hspace{#1pt}}

\newcommand{\R}{\mathbb{R}}
\newcommand{\N}{\mathbb{N}}

\newcommand{\J}{\mathcal{J}}
\newcommand{\Jhat}{\hat{\mathcal{J}}}

\newcommand{\cJhatn}{{{\Jhat_\red}}}

\newcommand{\HH}{\mathcal{H}}
\newcommand{\cHhatn}{\hat{\mathcal{H}}_{\red,\mu}}

\newcommand{\HHhat}{\hat{\mathcal{H}}}

\newcommand{\pr}{\textnormal{pr}}
\newcommand{\du}{\textnormal{du}}

\newcommand{\dred}[1]{d_{#1}}

\newcommand{\cont}[1]{\gamma_{#1}}

\newcommand{\bformd}{a_{\mu}}
\newcommand{\lformd}{l_{\mu}}
\newcommand{\kformd}{k_{\mu}}
\newcommand{\jformd}{j_{\mu}}
\newcommand{\resd}{r_{\mu}}

\newcommand{\Params}{\mathcal{P}}

\newcommand{\red}{{r}}

\newcommand{\Proj}{\mathrm{P}}

\makeatletter
\newcommand{\labeltext}[2]{%
  \@bsphack
  \csname phantomsection\endcsname 
  \def\@currentlabel{#1}{\label{#2}}%
  \@esphack
}
\makeatother

\allowdisplaybreaks

\title[A Newton TR-RB approximation for PDE-constrained optimization ]{An adaptive projected Newton non-conforming dual approach for trust-region reduced basis approximation of PDE-constrained parameter optimization}
\thanks{The authors acknowledge funding by the Deutsche Forschungsgemeinschaft (DFG) for the project {\em Localized Reduced Basis Methods for PDE-constrained Parameter Optimization}
under contracts OH 98/11-1; SCHI 1493/1-1; VO 1658/6-1. T. Keil, M. Ohlberger and F. Schindler acknowledge funding by the DFG under Germany’s Excellence Strategy EXC 2044 390685587, Mathematics M\"unster: Dynamics -- Geometry -- Structure.
}

\author[S.~Banholzer]{Stefan Banholzer}

\author[others]{Tim Keil}

\author[ ]{Luca Mechelli}

\author[ ]{Mario Ohlberger}

\author[ ]{Felix Schindler}

\author[ ]{Stefan Volkwein}

\address[S.~Banholzer, L.~Mechelli, S.~Volkwein]{Department of Mathematics and Statistics, Universit\"at Konstanz, Universit\"atsstr.~10, D-78457 Konstanz, Germany.} 
\email{{\tt \{stefan.banholzer,luca.mechelli,stefan.volkwein\}@uni-konstanz.de}}

\address[T.~Keil, M.~Ohlberger, F.~Schindler]{Mathematics M\"unster, Westf\"alische Wilhelms-Universit\"at M\"unster, Einsteinstr.~62, D-48149 M\"unster, Germany.}
\email{{\tt \{tim.keil,mario.ohlberger,felix.schindler\}@uni-muenster.de}}



\begin{document}

\begin{abstract}
In this contribution we device and analyze improved variants of the non-conforming dual approach for trust-region reduced basis (TR-RB) approximation of PDE-constrained parameter optimization that has recently been introduced in [Keil et al.. A non-conforming dual approach for adaptive Trust-Region Reduced Basis approximation of PDE-constrained optimization. arXiv:2006.09297,  2020]. 
The proposed methods use model order reduction techniques for parametrized PDEs to significantly reduce the computational demand of parameter optimization with PDE constraints in the context of large-scale or multi-scale applications.
The adaptive TR approach allows to localize the reduction with respect to the parameter space
along the path of optimization without wasting unnecessary resources in an offline phase. 
The improved variants employ projected Newton methods to solve the local optimization problems within each TR step to benefit from high convergence rates. This implies new strategies in constructing the RB spaces, together with an estimate for the approximation of the hessian.
Moreover, we present a new proof of convergence of the TR-RB method based on infinite-dimensional arguments, not restricted to the particular case of an RB approximation and provide an a posteriori error estimate for the approximation of the optimal parameter.
Numerical experiments demonstrate the efficiency of the proposed methods.
\end{abstract}

\maketitle

\section*{Introduction}
\label{sec:introduction}
 
Parametric PDE-constrained optimization problems are of interest in many fields, such as geology, chemistry and engineering. Although the PDE model effectively describes the behavior of the system, these models may lead to difficulties when computing an optimal solution with respect to a given cost. First, it may not be guaranteed that there exists a unique optimum, due to the fact that the problem may not be (strictly) convex. Second, discretizing the PDE by, e.g., Finite Element (FE) or Finite Volume methods leads to high dimensional full order models (FOM) which might be arbitrarily costly to solve.
The latter led to an extensive research activity over the last two decades, particular remedies include mesh adaptivity and/or model order reduction (MOR), see \cite{BKR2000,MR2556843,Clever2012,MR2905006,MR1887737,MR2892950} and \cite{BCOW2017,HRS2016,QMN2016}, respectively.\\[.4em]
{\em Model order reduction for PDE-constrained optimization.}
MOR techniques are a broad family of methods used to reduce the computational complexity of a given system, by exploiting its underlying structure and by building a reduced order model (ROM). Among these, the Reduced Basis method (RB) is particularly suited for parameter-dependent problems. This projection-based technique consists in reconstructing an approximation to the solution manifold of the PDE in a low-dimensional linear space, spanned by given solutions (snapshots) for carefully selected parameters.
One approach to construct a ROM (the so called \emph{offline phase}) is to employ a goal-oriented greedy algorithm based on a posteriori error estimates on the error between FOM and ROM quantities, resulting in quasi-optimally selected snapshots \cite{BCD+2011,HAA2017}.
Alternatively, the ROM can be built by means of a proper orthogonal decomposition (POD) in the method of snapshots; see \cite{GV17} and the references therein.
Once the ROM is built, it can be evaluated quickly (the so-called online phase).
There exists a large amount of literature using such reduced order surrogate models for optimization methods. A posteriori error estimates for reduced order approximation of linear-quadratic parametric optimization (and optimal control) problems were studied, e.g., in \cite{Dede2012,DH2015,GK2011,KTV13,NRMQ2013,OP2007}. In particular, in \cite{DH2015,KTV13} the authors show a posteriori error estimates also for the error between the optimal parameter/control and the approximate one. Although the standard offline/online decomposition is a viable approach for parametric optimization problems, its performance suffers when the dimension of the parameter space increases significantly. For very high dimensional parameter sets, simultaneous parameter and state reduction can be considered \cite{HO2015,LWG2010}.
To speed-up the process, it is advantageous to follow the optimization pattern and compute only locally accurate RB models; see, e.g. \cite{BMV2020,GHH2016,ZF2015}. In this context, localized RB methods, based on efficient localized a posteriori error control
and online enrichment, are particularly well-suited \cite{BEOR2016,BIORSS21,OSS2018,OS2015,OS2017}. With respect to the above mentioned works, we are interested in a different, but related approach, which is based on a trust-region (TR) method. \\[.4em]
{\em Trust-Region reduced order models for parametric PDE-constrained optimization.}
TR approaches are widely used in optimization, thanks to their robust behavior, which ensures global convergence. The key idea is to define a local approximation of the nonlinear objective, which allows using faster optimization tools; cf.~\cite{CGT00,NW06}. Obviously, the accuracy of the surrogate model has to be monitored during the TR iterations and possible updates have to be considered. A well-established method for MOR is the TR-POD algorithm \cite{AFS00,RoggTV17}. Furthermore, in \cite{QGVW2017} a TR-RB algorithm is presented for PDE-constrained optimization problems with unbounded parameter sets. This method is based on  \cite{YM2013}, where necessary and sufficient conditions are given to guarantee the convergence of the TR method. In this case, the TR is defined accordingly to the a posteriori error estimate for the cost functional. In \cite{KMOSV20}, the TR-RB method of \cite{QGVW2017} is extended to the case of constrained parameter sets and further improved regarding its convergence. \\[.4em]
{\em Main results.}
In this contribution we present several significant advances for the adaptive 
TR-RB optimization method presented in \cite{KMOSV20}:
\begin{itemize}
\item We propose higher order TR-RB methods using the projected Newton method to solve the TR sub-problems. The gradient and hessian of the optimization cost functional are approximated using the non-conforming dual (NCD) approach;
\item we provide efficiently computable a posteriori error estimates for the ROM error in reconstructing the FOM hessian and the optimal parameter;
\item we present a new proof of convergence of the TR-RB method based on infinite-dimensional arguments, not restricted to the particular case of an RB approximation;
\item we devise a new adaptive enrichment strategy for the progressive construction of RB spaces, including rigorous conditions for skipping enrichment to ensure the smallest possible ROM dimension;
\item we demonstrate in numerical experiments that our new TR-RB methods outperform existing
approaches for large scale optimization problems in well defined benchmark problems.  \\[-.6em]
\end{itemize} 
{\em Organization of the article.}
In Section~\ref{sec:problem} we introduce the PDE-constrained optimization problem and the necessary and sufficient optimality condtions to characterize local minimizers.
In Section~\ref{sec:mor}, we derive the FOM and ROM and furthermore state the a posteriori error estimates to certify the ROM, with particular focus on the approximation of the FOM hessian and the optimal parameter.
The improved adaptive TR-RB algorithm is introduced in Section~\ref{sec:TRRB_and_adaptiveenrichment}, where also the convergence analysis and adaptive Taylor-based enrichment strategy are carried out.
Finally, the numerical experiments, in which we compare the algorithm to selected state of the art optimization methods from the literature, are illustrated in Section~\ref{sec:num_experiments}.

\section{Problem formulation}
\label{sec:problem}

Given a real-valued Hilbert space $V$ with inner product $(\cdot \,,\cdot)$ and its induced norm $\|\cdot\|$, we are interested in efficiently approximating
PDE-constrained parameter optimization of a quadratic continuous functional $\J: V \times \Params \to \R$,
where the compact and convex admissible parameter set $\Params \subset \R^P$, with $P \in \N$ is considered to describe bilateral box constraints,
i.e.,
\[ 
\Params:= \left\{\mu\in\mathbb{R}^P\,|\,\mu_\mathsf{a} \leq \mu \leq \mu_\mathsf{b} \right\} \subset \R^P,
\] 
for given parameter bounds $\mu_\mathsf{a},\mu_\mathsf{b}\in\mathbb{R}^P$, where {``$\leq$''} has to be understood component-wise. 
To be more precise, we consider the minimization problem
{\color{white}
	\begin{equation}
	\tag{P}
	\label{P}
	\end{equation}
}\vspace{-51 pt} 
\begin{subequations}\begin{align}
	& \min_{\mu \in \Params} \J(u_\mu, \mu),
	&&\text{with } \J(u, \mu) = \Theta(\mu) + j_\mu(u) + k_\mu(u, u),
	\tag{P.a}\label{P.argmin}\intertext{%
		subject to $u_\mu \in V$ being the solution of the \emph{state -- or primal -- equation}
	}
	&a_\mu(u_\mu, v) = l_\mu(v) &&\text{for all } v \in V,
	\tag{P.b}\label{P.state}
	\end{align}\end{subequations}%
\setcounter{equation}{0}
where $\Theta \in \Params \to \mathbb{R}$ denotes a parameter functional. 
For each admissible parameter $\mu \in \Params$, $a_\mu: V \times V \to \R$ denotes a continuous and coercive bilinear form,
$l_\mu, j_\mu: V \to \R$ are continuous linear functionals and $k_\mu: V \times V \to \R$ denotes a continuous symmetric bilinear form. 
The primal residual of \eqref{P.state} is key for the optimization as well as for a posteriori error estimation. 
We define for given $u \in V$, $\mu \in \Params$, the primal residual $r_\mu^\pr(u) \in V'$ associated with \eqref{P.state} by
\begin{align}
r_\mu^\pr(u)[v] := l_\mu(v) - a_\mu(u, v) &&\text{for all }v \in V.
\label{eq:primal_residual}
\end{align}
\begin{remark}
	The Lagrange functional for \eqref{P} is given by $\mathcal{L}(u,\mu,p) = \J(u,\mu) + r_\mu^\pr(u)[p]$ for $(u,\mu)\in V\times\Params$ and for $p\in V$. In particular we have $\J(u_\mu,\mu) = \mathcal{L}(u_\mu,\mu,p)$ for all $p\in V$. 
\end{remark} 
A standard assumption for the efficient employment of RB methods is the parameter separability from $V$, which we assume in this work. 
For applications where this assumption does not hold, so-called empirical interpolation (EI) techniques \cite{BMNP2004,CS2010,DHO2012} can be utilized. 
\begin{assumption}[Parameter-separability]
\label{asmpt:parameter_separable}
		For $\Xi^a, \Xi^l, \Xi^j, \Xi^k \in \N$, we assume $a_\mu$, $l_\mu$, $j_\mu$, $k_\mu$ to be parameter separable with non-parametric components
		$a_i: V \times V \to \R$ for $1 \leq i \leq \Xi^a$, $l_i \in V'$ for $1 \leq i \leq \Xi^l$, $j_i \in V'$ for $1 \leq i \leq \Xi^j$
		and $k_i: V \times V \to \R$ for $1 \leq i \leq \Xi^k$,
		and respective parameter functionals $\theta_i^a, \theta_i^l, \theta_i^j, \theta_i^k \in \Params'$, such that\\[-1.2em]
		\begin{align*}
		a_\mu(u, v) &= \sum_{i = 1}^{\Xi^a} \theta_i^a(\mu)\, a_i(u, v)
		, &&&&& l_\mu(v) &= \sum_{i = 1}^{\Xi^l} \theta_i^l(\mu)\, l_i(v),
		\end{align*}
		and analogously for $j_\mu$ and $k_\mu$.
\end{assumption}
Parameter separability also holds for the primal residual, the cost functional as well as all other linear dependent quantities  in this work.

Gradient-based solution methods for problems of type \eqref{P} require information about first-order directional derivatives of the cost functional $\J$. 
If second-order derivatives are available, more advanced optimization routines can be applied which generally yields higher local convergence rates.

\subsection{Notation for differentiability}
\label{sec:differentiability}
Assuming the objective functional $\J: V \times \Params \to \R$ to be Fr\'echet differentiable w.r.t. $\mu \in \Params$, we define
the Fr\'echet derivative of $\J$ w.r.t.~its second argument in the direction of $\nu \in \mathbb{R}^P$
by $\partial_\mu \J(u, \mu) \cdot \nu$ (noting that the dual space of $\mathbb{R}^P$ is itself).
Moreover, we refer to $\partial_\mu \J(u, \mu)$ as the derivative w.r.t.~$\mu$ 
and  
for $u \in V$, $\mu \in \Params$ we denote 
the \emph{partial derivative} of $\J(u, \mu)$ w.r.t.~the $i$-th
component of $\mu$ by $\partial_{\mu_i}\J(u, \mu)$ for $1 \leq i \leq P$.
Note that $\partial_{\mu_i}\J(u, \mu) = \partial_\mu \J(u, \mu)\cdot e_i $,
where $e_i \in \R^P$ denotes the $i$-th canonical unit vector.
Furthermore, we denote the \emph{gradient} of $\J$ w.r.t.~its second argument -- the vector of components
$\partial_{\mu_i}\J(u, \mu)$ -- by the operator $\nabla_\mu \J: V \times \Params \to \R^P$.
Similarly, if $\J$ is Fr\'echet differentiable w.r.t.~each $u \in V$,
for each $u \in V$ and each $\mu \in \Params$ there exists a bounded linear functional $\partial_u \J(u, \mu)\in V'$,
such that the Fr\'echet derivative of $\J$ w.r.t.~its first argument in any direction $v \in V$ is given by $\partial_u \J(u, \mu)[v]$.
We refer to $\partial_u \J(u, \mu)$ simply as the derivative w.r.t.~$u$.
If $\J$ is twice Fr\'echet differentiable w.r.t.~each $\mu \in \Params$, we denote its \emph{hessian} w.r.t.~its second argument by the operator $\HH_\mu \J: V \times \Params \to \R^{P \times P}$. Finally, we denote the total derivative w.r.t.~$\mu_i$ by $d_{\mu_i}$, i.e. 
$d_{\mu_i} \J(u_{\mu}, \mu) = \partial_{\mu_i} \J(u_{\mu}, \mu) +  \partial_{u} \J(u_{\mu}, \mu)[ d_{\mu_i}u_{\mu}]$.
We treat $a$, $l$, $j$ and $k$ in a similar manner, although, for notational compactness, we indicate their parameter-dependency by a subscript for compactness and refer to \cite{KMOSV20} for further details. 

\begin{assumption}[Differentiability of $a$, $l$ and $\J$]
	\label{asmpt:differentiability}
  We assume $a_{\mu}$, $l_{\mu}$ and $\J$ to be twice continuously Fr\'{e}chet differentiable w.r.t.~$\mu$. This obviously requires that all parameter-dependent coefficient functions in Assumption~{\rm{\ref{asmpt:parameter_separable}}} are twice continuously differentiable as well. We also require all $\mu$-dependent functions to have locally Lipschitz-continuous second derivatives (for locally quadratic convergence of the projected Newton method).
\end{assumption}

For the continuous and coercive bilinear form $a_\mu(\cdot\,,\cdot)$
, we can define the bounded solution map $\mathcal{S}:\Params \to V$, $\mu \mapsto u_\mu=: \mathcal S(\mu)$, where $u_\mu$ is the unique solution to \eqref{P.state} for a given $\mu\in\Params$. 
The Fr\'echet derivatives of $\mathcal{S}$ have been used for RB methods for constructing Taylor RB spaces (see \cite{HAA2017}) and for deriving optimality conditions for \eqref{P} (see \cite{HPUU2009,Tro2010}).
\begin{proposition}[Fr\'{e}chet derivative of the solution map]
	\label{prop:solution_dmu_eta}
	Considering the solution map $\mathcal S:\Params \to V$, $\mu \mapsto u_\mu= \mathcal{S}(\mu)$, its Fr\'{e}chet derivative $d_{\nu} u_\mu= \mathcal{S}'(\mu)\cdot\nu \in V$ w.r.t.~a direction $\nu\in\mathbb{R}^P$ is the unique solution of
	\begin{align}\label{eq:primal_sens}
	a_\mu(d_{\nu}u_\mu, v) = \partial_\mu r_\mu^\pr(u_\mu)[v] \cdot \nu &&\text{for all } v \in V.
	\end{align}
\end{proposition}
\begin{proof}
	We refer, e.g., to \cite{HPUU2009,Tro2010} for the proof of this result.
\end{proof}

\subsection{Optimal solution and optimality conditions}
\label{sec:first_order_optimality_conditions}
Existence of an optimal solution to the non-convex problem \eqref{P} follows from \cite[Theorem~1.45]{HPUU2009}.
Using first- and second-order optimality conditions we can characterize local optimal solutions. Throughout the paper a bar indicates (local) optimality.
\begin{proposition}[First-order necessary optimality conditions]
	\label{prop:first_order_opt_cond}
  Let $(\bar u, \bar \mu) \in V \times \Params$ be a local optimal solution to \eqref{P}.
  Moreover, let Assumption~{\rm\ref{asmpt:differentiability}} hold true.
  Then there exists an associated unique Lagrange multiplier $\bar p\in V$ such that the following first-order necessary optimality conditions hold.
	\begin{subequations}
		\label{eq:optimality_conditions}
		\begin{align}
		r_{\bar \mu}^\pr(\bar u)[v] &= 0 &&\text{for all } v \in V,
		\label{eq:optimality_conditions:u}\\
		\partial_u \J(\bar u,\bar \mu)[v] - a_\mu(v,\bar p) &= 0 &&\text{for all } v \in V,
		\label{eq:optimality_conditions:p}\\
		(\partial_\mu \J(\bar u,\bar \mu)+\nabla_{\mu} r^\pr_{\bar\mu}(\bar u)[\bar p]) \cdot (\nu-\bar \mu) &\geq 0 &&\text{for all } \nu \in \Params. 
		\label{eq:optimality_conditions:mu}
		\end{align}	
	\end{subequations}
\end{proposition}
\begin{proof} We refer to \cite[Cor. 1.3]{HPUU2009} for a proof. 
\end{proof}
Note that \eqref{eq:optimality_conditions:u} corresponds to the state equation \eqref{P.state}.
From \eqref{eq:optimality_conditions:p} we deduce the so-called \emph{adjoint -- or dual -- equation}  
\begin{align}
a_\mu(q, p_\mu) = \partial_u \J(u_\mu, \mu)[q]
= j_\mu(q) + 2 k_\mu(q, u_\mu)&&\text{for all } q \in V,
\label{eq:dual_solution}
\end{align}
with solution $p_{\mu} \in V$ for a fixed $\mu \in \Params$, given the solution $u_\mu \in V$ to the state equation \eqref{P.state}.
We note that \eqref{eq:dual_solution} holds for quadratic $\J$ as in \eqref{P.argmin}. 
From \eqref{eq:optimality_conditions:p} we observe that the variable $\bar p$ of the optimal triple solves the dual equation \eqref{eq:dual_solution} for $\bar \mu$.
Similarly to the primal solution, we introduce the dual solution map $\mathcal A:\Params \to V$, $\mu \mapsto p_\mu := \mathcal A(\mu)$,
where $p_\mu$ is the solution of \eqref{eq:dual_solution} for the parameter $\mu$. Note that $\mathcal A$ is well-defined, because the bilinear form $a_{\mu}(\cdot \, , \cdot)$ is continuous and coercive. Moreover, $\bar p = p_{\bar \mu}$.
For given $u, p \in V$, we also introduce the dual residual $r_\mu^\du(u, p) \in V'$ associated with \eqref{eq:dual_solution} by
\begin{align}
r_\mu^\du(u, p)[q] := j_\mu(q) + 2k_\mu(q, u) - a_\mu(q, p)&&\text{for all }q \in V.
\label{eq:dual_residual}
\end{align}
Furthermore, from the dual equation \eqref{eq:dual_solution}, we obtain the following formulation for the dual sensitivities.
\begin{proposition}[Fr\'echet derivative of the dual solution map]
	\label{prop:dual_solution_dmu_eta}
	Considering the dual solution map $\mathcal A:\Params \to V$, $\mu \mapsto p_\mu= \mathcal{A}(\mu)$, we denote its directional derivative w.r.t.~a direction $\nu \in \Params$ by $d_{\nu} p_\mu= \mathcal{A}'(\mu)\cdot \nu  \in V$, which is given as the solution of
	\begin{equation} \label{eq:dual_sens}
	\begin{split}
	a_\mu(q, d_{\nu} p_\mu) &= -\partial_\mu a_\mu(q, p_\mu)\cdot \nu + d_\mu \partial_u \J(u_\mu, \mu)[q] \cdot \nu\\
	&
	= \partial_\mu r_\mu^\du(u_\mu, p_\mu)[q] \cdot \nu + 2 k_\mu(q, d_{\nu}u_\mu)
	\end{split}
	\end{equation}
	for all $q \in V$, where the latter equality holds for quadratic $\J$ as in \eqref{P.argmin}.
\end{proposition}
\begin{proof}
	For a proof we refer to \cite{HPUU2009,Tro2010}, for instance.
\end{proof}
We introduce the reduced functional $\Jhat: \Params\mapsto \mathbb{R}$, $\mu\mapsto \Jhat(\mu) := \J(u_\mu, \mu)= \J( \mathcal S(\mu),\mu)$.
Then problem \eqref{P} is equivalent to the so-called reduced problem
\begin{align}
\min_{\mu \in \Params} \Jhat(\mu).
\tag{$\hat{\textnormal{P}}$}\label{Phat}
\end{align}
In contrast to \eqref{P}, problem \eqref{Phat} has only inequality constraints, but no equality ones. Using definitions and notations from above we can compute first-order derivatives of $\Jhat$ by means of its gradient $\nabla_{\mu}\Jhat: \Params \to \R^P$.
\begin{proposition}[Gradient of $\Jhat$]
	\label{prop:grad_Jhat}
	For given $\mu\in\Params$, the gradient of $\Jhat$, $\nabla_\mu\Jhat: \Params \to \R^P$, is given by
	\begin{align*}
	\nabla_{\mu}\Jhat(\mu) &
	= \nabla_\mu \Theta(\mu) + \nabla_\mu j_\mu(u_\mu)
	+  \nabla_\mu k_\mu(u_\mu, u_\mu) + \nabla_{\mu}r_\mu^\pr(u_{\mu})[p_{\mu}].
	\end{align*}
\end{proposition}
\begin{proof}
	This follows from \eqref{eq:primal_residual}, \eqref{eq:primal_sens}, \eqref{eq:dual_solution} and \eqref{P.argmin}, cf.~\cite{HPUU2009}.
\end{proof}
\begin{remark}
	The proof of Proposition~{\rm{\ref{prop:grad_Jhat}}} relies on the fact that both $u_\mu$ and $p_\mu$ belong to the same space $V$; cf.~\emph{\cite{HPUU2009}}.
	In particular, for any $\mu\in\Params$, we have $\nabla_\mu \Jhat(\mu) = \nabla_\mu \mathcal L(u_\mu,\mu,p_\mu)$.
\end{remark}
For $\bar\mu$ satisfying the first-order necessary optimality conditions \eqref{eq:optimality_conditions}, we have that $\bar \mu$ is a stationary point of the cost functional $\Jhat$.
Thus, $\bar \mu$ can be either a local minimum, a saddle point or a local maximum of the cost
functional $\Jhat$ (and obviously the same relationship occurs between $(\bar u,\bar \mu)$ and $\J$).
We thus consider second-order sufficient optimality conditions in order to characterize local minima of the functional $\Jhat$, requiring its hessian.
\begin{proposition}[Hessian of $\Jhat$]
	\label{prop:hessian_Jhat}
	The hessian of $\Jhat$, $\HHhat_\mu := \HH_\mu \Jhat: \Params \to \R^{P \times P}$, is determined by its application to a direction $\nu\in \mathbb{R}^P$, given by
	\begin{align*}
	\HHhat_\mu(\mu) \cdot \nu
	= \nabla_\mu\Big(&\partial_u \J(u_\mu, \mu)[d_{\nu}u_\mu] + r_\mu^\pr(u_\mu)[d_{\nu}p_\mu] - a_\mu(d_{\nu}u_\mu, p_\mu)
	\\
	&+\big(\partial_\mu \J(u_\mu, \mu)+ \partial_\mu r_\mu^\pr(u_\mu)[p_\mu]\big)\cdot \nu\Big),
	\end{align*}
	where $u_\mu, p_\mu \in V$ denote the primal and dual solutions, respectively. 
	For a quadratic $\J$ as in \eqref{P.argmin} the above formula simplifies to
	\begin{align}
	\HHhat_\mu(\mu) \cdot \nu
	= & \nabla_\mu\Big(j_\mu(d_{\nu}u_\mu) + 2k_\mu(d_{\nu}u_\mu, u_\mu) + l_\mu(d_{\nu}p_\mu) -a_\mu(u_\mu, d_{\nu}p_\mu) \notag\\
	& - a_\mu(d_{\nu}u_\mu, p_\mu) 
	+\big(\partial_\mu \J(u_\mu, \mu)+ \partial_\mu l_\mu(p_\mu)- \partial_\mu a_\mu(u_\mu, p_\mu)\big)\cdot \nu\Big).
	\notag
	\end{align}
\end{proposition}
\begin{proof}
	See, e.g., \cite{HPUU2009} for the first part. The second one follows from a direct computation.
\end{proof}
\begin{proposition}[Second-order sufficient optimality conditions]\label{prop:second_order}
	Let Assumption~{\rm{\ref{asmpt:differentiability}}} hold true.
	Suppose that $\bar \mu\in \Params$ satisfies the first-order necessary optimality conditions \eqref{eq:optimality_conditions}.
	If $\HHhat_\mu(\bar \mu)$ is positive definite on the \emph{critical cone} $\mathcal C(\bar\mu)$ at $\bar\mu\in\Params$, i.e.,
	if $\nu \cdot(\HHhat_\mu(\bar \mu)\cdot \nu) > 0$ for all $\nu\in\mathcal C(\bar\mu)\setminus\{0\}$, with
	\begin{align*}
	\mathcal C(\bar\mu):= \big\{\nu\in\mathbb{R}^P\,\big|\, \exists \mu\in\Params,\,c_1>0: \nu = c_1(\mu-\bar \mu),\, \nabla_\mu \Jhat(\bar\mu)\cdot \nu = 0  \big\},
	\end{align*}
	then $\bar \mu$ is a strict local minimum of \eqref{Phat}.
\end{proposition} 
\begin{proof}
	For this result we refer to \cite{CasTr15,NW06}, for instance.
\end{proof}
\begin{remark}
For so-called small residual problems (i.e, $\|\partial_u \J(\bar u,\bar \mu)\|_V$ is small) one can ensure that the second-order sufficient optimality conditions hold. The proof is analogous to \cite[Section 3.3]{Vol2001}.
\end{remark}

\section{High dimensional discretization and model order reduction}
\label{sec:mor}
To discretize the optimization problem (\ref{P}) and the corresponding derivatives of the cost functional 
we use a classical Ritz-Galerkin projection onto a finite, but possibly high dimensional finite element space $V_h \subset V$. 
Based on this FOM we then define a ROM using the reduced basis method with possibly different reduced primal and dual 
state spaces as well as different reduced spaces for the primal and dual sensitivity equations. 
Since the resulting ROM will in general not be equivalent to a Ritz-Galerkin projection of the FOM onto a reduced space $V_\red \subset V_h$,
we follow the approach from \cite{KMOSV20}, to define a non-conforming dual (NCD) corrected ROM.  

\subsection{Full order model} 
\label{sec:problem_fom}
Assuming $V_h \subset V$ to be a finite-dimensional subspace, we define a Ritz-Galerkin projection of (\ref{P}) onto $V_h$ by
considering, for each $\mu \in \Params$, the solution $u_{h, \mu} \in V_h$ of the \emph{discrete primal equation}
\begin{align}
\bformd(u_{h, \mu}, v_h) = \lformd(v_h) &&\text{for all } v_h \in V_h,
\label{eq:state_h}
\end{align}
which gives $\resd^\pr(u_{h, \mu})[v_h] = 0$ for all $v_h \in V_h$, $\mu \in \Params$.
We also define, for each $\mu \in \Params$, the solution $p_{h, \mu} \in V_h$ of the \emph{discrete dual equation}
\begin{align}
\bformd(q_h, p_{h, \mu}) = \partial_u \J(u_{h, \mu}, \mu)[q_h] = \jformd(q_h) + 2 \kformd(q_h, u_{h, \mu}) &&\forall q_h \in V_h,
\label{eq:dual_solution_h}
\end{align}
which results in $\resd^\du(u_{h, \mu}, p_{h, \mu})[q_h] = 0$ for all $q_h \in V_h$, $\mu \in \Params$. Similarly, the \emph{discrete primal sensitivity equations} for solving for $d_{\nu} u_{h, \mu} \in V_h$ as well as
\emph{discrete dual sensitivity equations} for solving for $d_{\nu} p_{h, \mu} \in V_h$ at any direction $\nu \in \Params$
follow directly analogue to Propositions \ref{prop:solution_dmu_eta} and \ref{prop:dual_solution_dmu_eta}.
Furthermore, instead of $\Jhat$ we define the \emph{discrete reduced functional}
\begin{align}
\Jhat_h(\mu) := \J(u_{h, \mu}, \mu) = \mathcal L(u_{h,\mu},\mu,p_h) && \text{for all } p_h\in V_h,
\label{eq:Jhat_h}
\end{align}
where $u_{h, \mu} \in V_h$ is the unique solution of \eqref{eq:state_h}, and we formulate the discrete optimization problem 
\begin{align}
\min_{\mu \in \Params} \Jhat_h(\mu).
\tag{$\hat{\textnormal{P}}_h$}\label{Phat_h}
\end{align}
Further, $\bar\mu_h$ denotes a locally optimal solution to \eqref{Phat_h} satisfying first- and second-order optimality conditions. 
\begin{remark}
	Since $u_{h,\mu}$ and $p_{h,\mu}$ belong to the same space $V_h$, Propositions~{\rm{\ref{prop:first_order_opt_cond}-{\rm\ref{prop:grad_Jhat}},{\rm\ref{prop:hessian_Jhat}}-\ref{prop:second_order}}} from Section~{\rm{\ref{sec:problem_fom}}} hold for the FOM as well, with all quantities replaced by their discrete counterparts.
\end{remark}
Analogously to Proposition \ref{prop:hessian_Jhat} we define a shorthand for the hessian of the discrete reduced functional as $\hat{\HH}_{h, \mu} := \HH_\mu \Jhat_h: \Params \to \R^{P \times P}$.
As usual in the context of RB methods, we eliminate the issue of ``truth'' by assuming that the high dimensional space $V_h$ is accurate enough to approximate the true solution.

\begin{assumption}[This is the ``truth'']
	\label{asmpt:truth}
	We assume that the primal discretization error $\|u_\mu - u_{h, \mu}\|$, the dual error $\|p_\mu - p_{h, \mu}\|$, 
	the primal sensitivity errors $\|d_{\mu_i} u_\mu - d_{\mu_i} u_{h, \mu}\|$ and the dual sensitivity errors $\|d_{\mu_i} p_\mu - d_{\mu_i} p_{h, \mu}\|$
	are negligible for all $\mu \in \Params$, $1 \leq i \leq P$.
\end{assumption}

To define a suitable ROM for the optimality system, we assume that we have computed problem adapted RB spaces $V_\red^\pr, V_\red^\du \subset V_h$, 
the construction of which is detailed in Section~\ref{sec:construct_RB}.
We stress here that $V_\red^\pr$ and $V_\red^\du$ might not coincide,
which implies the use of the NCD-corrected approach for reducing the optimality system \eqref{eq:optimality_conditions}.

\subsection{NCD-corrected reduced order model}
\label{sec:rom}
Given problem adapted RB spaces $V_\red^\pr, V_\red^\du \subset V_h$ of low dimension $n := \dim V_\red^\pr$ and $m := \dim V_\red^\du$ we obtain the reduced versions for the optimality system as follows:
\begin{subequations}
	\label{eq:optimality_conditionsRB}
	\begin{itemize}
		\item RB approximation for \eqref{eq:optimality_conditions:u}: For each $\mu \in \Params$ the primal variable $u_{\red, \mu} \in V_\red^\pr$ 
		of the \emph{RB approximate primal equation} is defined through
		\begin{align}
		\bformd(u_{\red, \mu}, v_\red) = \lformd(v_\red) &\qquad \text{for all } v_\red \in V_\red^\pr.
		\label{eq:state_red}
		\end{align}
		\item RB approximation for \eqref{eq:optimality_conditions:p}: For each $\mu \in \Params$, $u_{\red, \mu} \in V_\red^\pr$ the dual/adjoint variable $p_{\red, \mu} \in V_\red^\du$ satisfies the \emph{RB approximate dual equation} 
		\begin{align}
		\bformd(q_\red, p_{\red, \mu}) = \partial_u \J(u_{\red, \mu}, \mu)[q_\red] = \jformd(q_\red) + 2 \kformd(q_\red, u_{\red, \mu}) &&\forall q_\red \in V_\red^\du.
		\label{eq:dual_solution_red}
		\end{align}
	\end{itemize}
\end{subequations}
Analogously to Proposition~\ref{prop:solution_dmu_eta},
we define the \emph{RB solution map} $\mathcal{S}_\red:\Params \to V_\red^\pr$ by $\mu \mapsto u_{\red, \mu}=:\mathcal{S}_\red(\mu)$ and analogously to
Proposition~\ref{prop:dual_solution_dmu_eta} the \emph{RB dual solution map} $\mathcal{A}_\red:\Params \to V_\red^\du$ by $\mu \mapsto p_{\red, \mu}=:\mathcal{A}_\red(\mu)$,
where $u_{\red, \mu}$ and $p_{\red, \mu}$ denote the primal and dual reduced solutions of \eqref{eq:state_red} and \eqref{eq:dual_solution_red}, respectively.
Note that, in general, \eqref{eq:dual_solution_red} is not the dual equation with respect to the optimization problem \eqref{eq:Jhat_red_corected},
cf.~\cite[Section~1.6.4]{HPUU2009}, which would only be true if $V^\du_\red = V^\pr_\red$.

There exist several ways to approximate \eqref{Phat_h} in a ROM.  The standard way is to simply replace all discretized quantities in the FOM by their respective reduced ones.
However, if the reduced primal and dual RB spaces do not coincide, this approach results in inexact gradient and hessian information of the model.
In \cite{KMOSV20}, it was shown that this also results in a loss of robustness in the optimization method.
Hence, we use a modified approach from \cite{KMOSV20}, i.e.~we define the \emph{NCD-corrected RB reduced functional} by 
\begin{align}
\cJhatn(\mu) := \mathcal{L}(u_{\red,\mu},\mu,p_{\red,\mu}) = \J(u_{\red, \mu}, \mu) + r_\mu^\pr(u_{\red,\mu})[p_{\red,\mu}]
\label{eq:Jhat_red_corected}
\end{align} 
with $u_{\red, \mu} \in V_\red^\pr$ and $p_{\red, \mu} \in V_\red^\du$ being the solutions of \eqref{eq:state_red} and \eqref{eq:dual_solution_red} for $\mu\in\Params$, respectively.
We then consider the \emph{RB reduced optimization problem} of finding a locally optimal solution $\bar \mu_\red$ of
\begin{align}
\min_{\mu \in \Params} \cJhatn(\mu).
\tag{$\hat{\textnormal{P}}_\red$}\label{Phat_\red}
\end{align}

As in Section \ref{sec:first_order_optimality_conditions} we require the gradient and hessian of $\cJhatn$, which
can be computed following \cite[Section~1.6.2]{HPUU2009}. 
\begin{proposition}[Gradient of the NCD-corrected RB reduced functional]
	\label{prop:true_corrected_reduced_gradient_adj}
	The $i$-th component of the true gradient of $\cJhatn$ is given by
	\begin{align*}
	\big(\nabla_\mu \cJhatn(\mu)\big)_i & = \partial_{\mu_i}\J(u_{\red,\mu},\mu) + \partial_{\mu_i}r_\mu^\pr(u_{\red,\mu})[p_{\red,\mu}+w_{\red,\mu}] \\ & \quad - \partial_{\mu_i}r^\du_\mu (u_{\red,\mu},p_{\red,\mu})[z_{\red,\mu}],
	\end{align*}
	where $u_{\red, \mu} \in V_\red^\pr$ and $p_{\red, \mu} \in V_\red^\du$ denote the RB approximate primal and dual solutions of \eqref{eq:state_red} and \eqref{eq:dual_solution_red}, $z_{\red,\mu} \in V_\red^\du$ solves
	\begin{equation}
	\label{A2:z_eq}
	a_\mu(z_{\red,\mu},q) = -r_\mu^\pr(u_{\red,\mu})[q] \quad \forall q\in V^\du_r
	\end{equation}
	and $w_{\red,\mu} \in V_\red^\pr$ solves
	\begin{equation}
	\label{A2:w_eq}
	a_\mu(v,w_{\red,\mu}) = r_\mu^\du(u_{\red,\mu},p_{\red,\mu})[v]-2k_\mu(z_{\red,\mu},v), \quad \forall v\in V^\pr_r.
	\end{equation}
\end{proposition}

We also define the derivatives of the maps $\mathcal{S}_\red$ and $\mathcal{A}_\red$ in direction $\nu \in \Params$ as the solutions $d_{\nu} u_{\red, \mu} \in V_\red^\pr$ and $d_{\nu} p_{\red, \mu} \in V_\red^\du$ of
\begin{align}
  a_\mu(d_{\nu} u_{\red, \mu}, v_\red) &= \partial_{\mu} r_\mu^\pr(u_{\red, \mu})[v_\red]\cdot\nu &&\text{for all } v_\red \in V_\red^\pr\label{eq:true_primal_reduced_sensitivity}
  \end{align}
  and
  \begin{equation}
  \label{eq:true_dual_reduced_sensitivity}
   \begin{aligned}
  a_\mu(q_\red, d_{\nu} p_{\red, \mu}) &= d_\mu\partial_u \J(u_{\red, \mu}, \mu)[q_\red]\cdot\nu  \\ & \quad  -\partial_\mu a_\mu(q_\red, p_{r, \mu})\cdot\nu 
  && \hspace{2pt}\text{for all } q_\red \in V_\red^\du,  
  \end{aligned}
  \end{equation}
respectively, analogously to Propositions \ref{prop:solution_dmu_eta} and \ref{prop:dual_solution_dmu_eta}, where the last equality holds for quadratic functionals as in \eqref{P.argmin}.
\begin{remark}
	For more accurate reduced derivatives of the solution maps in \eqref{eq:true_primal_reduced_sensitivity} and \eqref{eq:true_dual_reduced_sensitivity}
	one could again commit a variational crime by introducing problem adapted RB spaces for the primal and dual sensitivities w.r.t.~all canonical directions,
	i.e $V_\red^{\pr,d_{\mu_i}}$ and $V_\red^{\du,d_{\mu_i}}$. These spaces would then consist of FOM snapshots of the respective derivatives,
	i.e.~solutions of \eqref{eq:primal_sens} and \eqref{eq:dual_sens}; cf.~{\rm\cite{KMOSV20}}.
	We do not follow this strategy here, since the computational demand for enriching all these spaces scales with the size of the parameter space and quickly becomes unfeasible for large scale applications.
\end{remark}
With the help of the reduced derivatives of the primal and dual solution maps, we can also compute the hessian of the NCD-corrected RB reduced functional; cf.~\cite[Section~1.6.4]{HPUU2009}. 
\begin{proposition}[Hessian of the NCD-corrected RB reduced functional]
  \label{prop:true_corrected_reduced_hessian}
  Given a direction $\nu\in\Params$, the evaluation of the hessian $\cHhatn$ of $\cJhatn$ is
  \[
  \begin{aligned}
  \cHhatn(\mu)\cdot\nu & = \nabla_\mu \left(j_\mu(d_\nu u_{\red,\mu})+2k_\mu(u_{\red,\mu},d_\nu u_{\red,\mu})  - a_\mu(d_\nu u_{\red,\mu},p_{\red,\mu}+w_{\red,\mu}) \right.\\
  & + r^\pr_\mu(u_{\red,\mu})[d_\nu p_{\red,\mu}+d_\nu w_{\red,\mu}] - 2 k_\mu (z_{\red,\mu},d_\nu u_{\mu,\red}) \\
  &  + a_\mu(z_{\red,\mu},d_\nu p_{\red,\mu}) - r^\du_\mu(u_{\red,\mu},p_{\red,\mu})[d_\nu z_{\red,\mu}]   \\
  & \left.+\partial_\mu (\J(u_{\red,\mu},\mu) + r^\pr_\mu(u_{\red,\mu})[p_{\red,\mu}+w_{\red,\mu}]- r_\mu^\du(u_{\red,\mu},p_{\red,\mu})[z_{\red,\mu}])\cdot \nu \right)
  \end{aligned}
  \]
  where $d_\nu u_{\red,\mu}, w_{\red,\mu} \in V^\pr_\red$ and $d_\nu p_{\red,\mu}, z_{\red,\mu} \in V^\du_r$ solve \eqref{eq:true_primal_reduced_sensitivity}, \eqref{A2:w_eq},
  \eqref{eq:true_dual_reduced_sensitivity} and \eqref{A2:z_eq}, respectively. Furthermore, $d_\nu z_{\red,\mu} \in V^\du_r$ solves
  \begin{equation}
  \label{z_eq_sens}
  a_\mu(d_\nu z_{\red,\mu},q) = -\partial_\mu(r_\mu^\pr(u_{\red,\mu})[q] + a_\mu(z_{\red,\mu},q))\cdot \nu + a_\mu(d_\nu u_{\red,\mu},q)
  \end{equation}
  for all $q\in V^\du_r$ and $w_{\red,\mu} \in V^\pr_r$ solves
  \begin{equation}
  \label{w_eq_sens}
  \begin{aligned}
  a_\mu(v,d_\nu w_{\red,\mu}) &= \partial_\mu (r_\mu^\du(u_{\red,\mu},p_{\red,\mu})[v]-2k_\mu(z_{\red,\mu},v)-a_\mu(v,w_{\red,\mu}))\cdot\nu \\
  & +2k_\mu(v,d_\nu u_{\red,\mu}-d_\nu z_{\red,\mu})-a_\mu(v,d_\nu p_{\red,\mu}), \quad \forall\, v\in V^\pr_r.
  \end{aligned}
  \end{equation}
  \end{proposition}

There exist multiple possibilities for deducing a reduced hessian.
As a straight forward hessian, it is also feasible to consider the FOM hessian from Proposition \ref{prop:hessian_Jhat}
and reducing it by replacing all FOM quantities by their respective reduced counterpart. 
While this approach may be a better approximation of the FOM hessian,
it is not the true hessian of the NCD-corrected functional which would result in a quasi-Newton type method.
In order to prevent an overload of the work at hand, we omit a further discussion of this approach.
However, we emphasize that also for this approach a posteriori error analysis is available.
We further remark that the computation of the true hessian $\cHhatn(\mu)$ can also be realized without the use of auxiliary functions $z_{\red, \mu}$ and $w_{\red, \mu}$ and their derivatives, respectively.
However, this results in having to compute second order derivatives of $u_{\red, \mu}$ and $p_{\red, \mu}$ which aggravates the computations and makes the hessian inefficiently callable from
an optimization point of view because the second direction can not be pulled out. Thus, we also do not follow this approach.

\subsection{A posteriori error analysis} 

\label{sec:a_post_error_estimates}
For controlling the accuracy of the reduced model, we require a posteriori error estimates of all reduced quantities. 
Assumption \ref{asmpt:parameter_separable} is the key for the efficient computation of reduced quantities because it enables to assemble FOM matrices offline.
In this section, we re-state all estimates that we need for the error aware TR-RB method, and shortly mention an a posteriori result for the hessian of the NCD-corrected
RB reduced functional. We also present a bound for the distance to the true solution of the optimization problem.
For any functional $l \in V_h'$ or bilinear form $a: V_h \times V_h \to \R$, we denote their respective dual or operator norms $\|l\|$ and $\|a\|$ by the continuity constants $\cont{l}$ and $\cont{a}$.
The same consideration applies for the norm $\|\cdot\|$ in $V_h'$ of the residuals.
For $\mu \in \Params$, we denote the coercivity constant of $\bformd$ w.r.t.~the $V_h$-norm by $\underline{\bformd} > 0$. 

For $v_h \in V_h$, we define the residuals of the equation in Proposition~\ref{prop:solution_dmu_eta} and Proposition~\ref{prop:dual_solution_dmu_eta} for the canonical directions by
\begin{align}
\resd^{\pr,d_{\mu_i}}(&u_{h, \mu}, d_{\mu_i} u_{h, \mu})[v_h] := \partial_{\mu_i} \resd^\pr(u_{h, \mu})[v_h] - \bformd(d_{\mu_i} u_{h, \mu}, v_h),
\label{sens_res_pr}\\
\resd^{\du,d_{\mu_i}}(&u_{h, \mu}, p_{h, \mu}, d_{\mu_i} u_{h, \mu}, d_{\mu_i} p_{h, \mu})[v_h] \nonumber\\
&:= \partial_{\mu_i} \resd^\du(u_{h, \mu}, p_{h, \mu})[v_h] + 2\kformd(v_h, d_{\mu_i} u_{h, \mu}) - \bformd(v_h, d_{\mu_i} p_{h, \mu}).
\label{sens_res_du}
\end{align}

We summarize known error estimates from the literature and refer to \cite{KMOSV20} for a detailed discussion and proofs.
 
\begin{proposition}[Upper error bound for the reduced quantities]	\label{prop:error_reduced_quantities}
	For $\mu \in \Params$, let $u_{h, \mu}, p_{h, \mu} \in V_h$ be solutions of \eqref{eq:state_h} and \eqref{eq:dual_solution_h} and let $u_{\red, \mu} \in V_\red^\pr$ be 
	a solution of \eqref{eq:state_red}. Furthermore, for $1 \leq i \leq P$, let $d_{\mu_i}u_{h, \mu}, d_{\mu_i}p_{h, \mu} \in V_h$ be
	the solutions of the discrete versions of \eqref{eq:primal_sens} and \eqref{eq:dual_sens} and
	let $d_{\mu_i}u_{\red, \mu} \in V_\red^{\pr,d_{\mu_i}}\!\!$ and $d_{\mu_i}p_{\red, \mu} \in V_\red^{\pr,d_{\mu_i}}\!\!$ be the solutions of \eqref{eq:true_primal_reduced_sensitivity} and \eqref{eq:true_dual_reduced_sensitivity}.  Then it holds
	\begin{enumerate}[(i)]
		\item
		$\|u_{h, \mu} - u_{\red, \mu}\| \leq \Delta_\pr(\mu) := \underline{\bformd}^{-1}\, \|\resd^\pr(u_{\red, \mu})\|$,
		\item 	$\|p_{h, \mu} - p_{\red, \mu}\| \leq \Delta_\du(\mu) := \underline{\bformd}^{-1}\big(2 \cont{\kformd}\;\Delta_\pr(\mu) + \|\resd^\du(u_{\red, \mu}, p_{\red, \mu})\|\Big)$,
		\item  $|\Jhat_h(\mu) - \cJhatn(\mu)| \leq \Delta_{\cJhatn}(\mu)
		:=  \Delta_\pr(\mu) \|\resd^\du(u_{\red, \mu}, p_{\red,\mu})\| + \Delta_\pr(\mu)^2 \cont{\kformd}$,
	\item $	\|d_{\mu_i}u_{h, \mu} - d_{\mu_i}u_{\red, \mu}\| \leq \Delta_{d_{\mu_i}\pr}(\mu)$,
	\item $\|d_{\mu_i}p_{h, \mu} - d_{\mu_i}p_{\red, \mu}\| \leq \Delta_{ d_{\mu_i}\du}(\mu)$,
	\end{enumerate}
	where
	\begin{align*}
	\Delta_{d_{\mu_i}\pr}(\mu) &:= \underline{\bformd}^{-1}\Big(\cont{d_{\mu_i} \bformd} \Delta_\pr(\mu) + \|\resd^{\pr,d_{\mu_i}}(u_{\red, \mu}, d_{\mu_i}u_{\red, \mu})\| \Big),\\
	\Delta_{ d_{\mu_i}\du}(\mu) &:= \underline{\bformd}^{-1}\Big(
	2 \cont{d_{\mu_i} \kformd} \; \Delta_\pr(\mu) +  \cont{d_{\mu_i} \bformd} \; \Delta_\du(\mu) + 2 \cont{\kformd}  \; \Delta_{d_{\mu_i}pr}(\mu) \\
	&\qquad\qquad\quad+ \| \resd^{\du,d_{\mu_i}}(u_{\red, \mu}, p_{\red, \mu}, d_{\mu_i}u_{\red, \mu}, d_{\mu_i}p_{\red, \mu}) \| \Big).
	\end{align*}
\end{proposition}	

We also provide an a posteriori error result for the hessian of the NCD-corrected functional.
We emphasize that (just as the sensitivity estimates $\Delta_{d_{\mu_i}\pr}(\mu)$ and $\Delta_{ d_{\mu_i}\du}(\mu)$) this error estimator is not a part of our TR-RB method.
The proof and a detailed definition is postponed to the appendix.

\begin{proposition}[Upper bound on the model reduction error of the hessian of the reduced output]
	\label{prop:hessian_Jhat_error_NCD} 
	For the hessian $\HHhat_{h,\mu}(\mu)$ of $\Jhat_h(\mu)$ and the true hessian 
	$\cHhatn(\mu)$ of the NCD-corrected functional from Proposition {\rm\ref{prop:true_corrected_reduced_hessian}},
	there exists an a posteriori error bound
	\begin{align*}
	\big|\HHhat_{h,\mu}(\mu) - \cHhatn(\mu)&\big| \leq \Delta_{{\HH}}(\mu) := \Big\| \big(\Delta_{\HH_{i,l}}(\mu)\big)_{i,l} \Big\|_2
	\end{align*}
	which is dependent on the estimators from Proposition \ref{prop:error_reduced_quantities}, except $\Delta_{\cJhatn}(\mu)$.
\end{proposition}

Following ideas from \cite{dihl15,KTV13}, we derive an error estimation for the optimal parameter consisting of the gradient and hessian of the FOM cost functional. This estimator relies on the following second-order condition for a strict local minima $\bar\mu_h$ of $\Jhat_h$, i.e.
\begin{align}
\label{coerc-cond}
\nu \cdot ( \HHhat_{h,\mu}(\bar \mu_h) \cdot \nu) \geq \lambda_\text{min}\left\|\nu\right\|_2^2 && \text{for all } \nu\in\mathcal{C}(\bar\mu_h)\setminus\left\{0\right\},
\end{align} 
where $\lambda_\text{min}$ is the smallest eigenvalue of $\HHhat_{h,\mu}(\bar\mu_h)$, since the parameter space is finite-dimensional. Note that \eqref{coerc-cond} is equivalent to the second-order sufficient optimality condition from Proposition~\ref{prop:second_order}. If \eqref{coerc-cond} holds true, we have that for any $\tilde{\lambda}$ such that $0<\tilde{\lambda}<\lambda_\text{min}$ there exists a radius $r(\tilde{\lambda})>0$ such that for all $\mu\in \mathcal{B}(\bar \mu_h, r(\tilde{\lambda}))$, the closed ball of radius $r(\tilde{\lambda})$ centered in $\bar\mu_h$, the following property holds:
\begin{align*}
\nu \cdot ( \HHhat_{h,\mu}(\mu) \cdot \nu) \geq \tilde{\lambda}\left\|\nu\right\|_2^2 && \text{for all } \nu\in\mathcal{C}(\bar\mu_h)\setminus\left\{0\right\}.
\end{align*} 

\begin{proposition}[Upper bound for optimal parameters with the full order model] \label{Prop:argmin}
Let Assumption~{\rm\ref{asmpt:truth}} be satisfied. Moreover, let $\bar\mu_h$ and $\bar \mu_\red$ be strict local minima for the optimization problems \eqref{Phat_h} and \eqref{Phat_\red}, respectively. If $\bar\mu_\red \in \mathcal{B}(\bar \mu_h, r(\lambda_\text{min}/2))$, then it holds 
\begin{equation}
\label{apo-est-parameters}
\| \bar \mu_h - \bar \mu_\red \|_2 \leq \Delta_\mu(\bar \mu_\red) := \frac{2}{\lambda_\text{min}} \left\| \zeta \right\|_2,
\end{equation}
where $\zeta= (\zeta_i)\in\mathbb{R}^P$ with
\[
\zeta_i:= \left\{ \begin{array}{ll} -\min(0,(\nabla\Jhat_h(\bar\mu_\red))_i) & \text{if } \bar\mu_{\red,i}=(\mu_\mathsf{a})_i \\
-\max(0,(\nabla\Jhat_h(\bar\mu_\red))_i) & \text{if } \bar\mu_{\red,i}=(\mu_\mathsf{b})_i   \\
-(\nabla\Jhat_h(\bar\mu_\red))_i & \text{otherwise}
\end{array} \right.
\]
for $i=1,\ldots,P$.	
\end{proposition}
\begin{proof}
Note that Assumption~\ref{asmpt:truth} implies that the distance between $\bar\mu_h$ of \eqref{Phat_h} and a strict local minimum $\bar\mu$ of $\Jhat$ satisfying \eqref{coerc-cond} is negligible, thus we can follow the proof of \cite[Theorem~3.4]{KTV13}.
\end{proof}
\begin{remark}
	\label{rem:mu_est_fom}
(1) Proposition~{\rm\ref{Prop:argmin}} requires the strong assumption that the FOM and RB models are accurate enough to have the parameters $\bar\mu_h$ and $\bar\mu_\red$ sufficiently close to a local minimum $\bar\mu$. 
	In {\rm\cite{dihl15}}, a sufficient condition based on the FOM gradient and hessian is given to guarantee this in case $\bar\mu_h,\bar\mu_\red\in \text{int }\Params$. \\
(2) Due to Proposition~{\rm\ref{Prop:argmin}}, we can estimate the distance to the optimal parameter $\bar\mu_h$ without explicitly computing it.
	Note that the computation of $\zeta$ is not costly for Algorithm~{\rm\ref{Alg:TR-RBmethod}}, since the FOM adjoint solution is available.
	The computation of $\lambda_{\text{min}}$ would require the evaluation of the FOM hessian, which is a costly procedure instead. This can be spead up with a cheap estimation of the eigenvalue.
	In {\rm\cite[Proposition~6]{dihl15}}, the authors utilize the smallest eigenvalue of the reduced-order hessian under suitable conditions.
	In our numerical tests, these conditions were never true, implying the inapplicability of the mentioned cheap estimate in our case.
	For the sake of completeness, let us mention that another technique is to compute $\lambda_{\text{min}}$ in advance on a grid in $\Params\subset\mathbb{R}^P$, when $P$ is sufficiently small.
	This approach can be even performed in parallel, since each eigenvalue computation is independent; cf. \cite[Section~6.4.1]{Trenz2017}.\\
(3) Due to the above-mentioned computational cost, we use estimate \eqref{apo-est-parameters} only as post-processing tool:
	once the TR-RB algorithm (cf. Section~{\rm\ref{sec:TRRB_and_adaptiveenrichment}}) has converged, we check if its solution is close enough to $\bar\mu_h$.
	If not, we decrease the stopping tolerance $\tau_{\text{\rm{FOC}}}$ (cf. Algorithm~{\rm\ref{Alg:TR-RBmethod}}) and continue with the algorithm.
\end{remark}

\section{The improved TR-RB Method}
\label{sec:TRRB_and_adaptiveenrichment}
Trust-region methods iteratively compute a first-order critical point of problem \eqref{P}.
For each outer iteration $k\geq 0$ of the TR method, we consider a model function $m^{(k)}$
as a cheap local approximation of the quadratic cost functional $\J$ in the so-called trust-region, which has radius $\delta^{(k)}$. We are therefore interested in solving the following constrained optimization sub-problem
\begin{equation}
\label{TRsubprob}
\begin{aligned}
\min_{s\in \mathbb{R}^P} m^{(k)}(s) \, \text{ subject to } & \|s\|_2 \leq \delta^{(k)},\, \widetilde{\mu}:= \mu^{(k)}+s \in\Params \\ & \text{ and } r_{\tilde{\mu}}^\pr(u_{\tilde{\mu}})[v]= 0 \, \text{ for all }  v\in V.
\end{aligned}
\end{equation}
Under suitable assumptions, problem \eqref{TRsubprob} admits a unique solution $\bar s^{(k)}$, which is used to compute the next outer TR iterate $\mu^{(k+1)} = \mu^{(k)} + \bar s^{(k)}$. 

\subsection{The projected Newton based TR-RB Method with optional enrichment}
\label{sec:TR}
Trust-region methods combined with MOR techniques have been extensively studied in, e.g.,  \cite{AFS00,BCB07,KMOSV20,QGVW2017}. Among these methods, we are interested in TR-RB algorithms. As discussed in \cite{KMOSV20}, it is advantageous to choose the NCD-corrected RB reduced functional as the model function,
i.e.~$m^{(k)}(\cdot)= \cJhatn^{(k)}(\mu^{(k)}+\cdot)$ for $k\geq 0$, where the super-index $(k)$ indicates that we use different RB spaces $V_\red^{*, (k)}$ in each iteration. 
We initialize the RB space with the starting parameter $u_{\mu^{(0)}}$, i.e.~$V^{\pr,(0)}_\red = \big\{u_{h,\mu^{(0)}}\big\}$ and $V^{\du,(0)}_\red = \big\{p_{h,\mu^{(0)}}\big\}$.
Like in \cite{KMOSV20}, we consider bilateral parameter constraints but employ a projected Newton method to solve \eqref{TRsubprob}, which has a faster local convergence compared to the projected BFGS, used in \cite{KMOSV20,QGVW2017}. We first state the TR-RB method suggested in \cite{KMOSV20,QGVW2017}, then we remark the further improvements introduced in addition to \cite{KMOSV20,QGVW2017}. The RB version of problem \eqref{TRsubprob} is 
\begin{equation}
\label{TRRBsubprob}
\min_{\widetilde{\mu}\in\Params} \cJhatn^{(k)}(\widetilde{\mu}) \quad \text{ s.t. } \quad  \frac{\Delta_{\Jhat}(\widetilde{\mu})}{\cJhatn^{(k)}(\widetilde{\mu})}\leq \varrho^{(k)},
\end{equation}
where $\widetilde{\mu}:= \mu^{(k)}+s$, the equality constraint $r_\mu^\pr(u_{\widetilde{\mu}})[v]= 0$ is hidden in the definition of $\cJhatn$
and the inequality constraints are concealed in the request $\widetilde{\mu}\in \Params$.
Due to the presence of bilateral constraints on the parameters, we introduce the projection operator $\Proj_\Params: \mathbb{R}^p\rightarrow \Params$ defined as
\begin{align*}
(\Proj_\Params(\mu))_i:= \left\{ \begin{array}{ll}
(\mu_\mathsf{a})_i & \text{if } \mu_i\leq (\mu_\mathsf{a})_i, \\
(\mu_\mathsf{b})_i & \text{if } \mu_i\geq (\mu_\mathsf{b})_i, \\
\mu_i & \text{otherwise}
\end{array} \right.  && \text{for } i=1,\ldots,P.
\end{align*}
The operator $\Proj_\Params$ is Lipschitz continuous with Lipschitz constant one; cf.~\cite{Kel99}. The additional TR constraint, instead, is treated with a backtracking technique; cf.~\cite{QGVW2017}. For solving \eqref{TRRBsubprob} at iteration $k$, the projected Newton method uses the approximated generalized Cauchy (AGC) point $\mu^{(k)}_\text{AGC}$ (cf. Definition~\ref{Def:AGC}) as warm start and generates a sequence $\{\mu^{(k,\ell)}\}_{\ell=1}^L$, where $L$ is the last Newton iteration.
In what follows, $\mu^{(k,1)}:= \mu^{(k)}_\text{AGC}$ and the TR iterate $\mu^{(k+1)}:=\mu^{(k,L)}$.
Throughout the paper the index $k$ refers to the current outer TR iteration, $\ell$ refers instead to the inner Newton iteration.
Note that $L$ may be different for each iteration $k$. To simplify the notation, we omit this dependence unless it is strictly necessary to specify it. We define
\begin{equation}
\label{eq:General_Opt_Step_point}
\mu^{(k,\ell)}(j):= \Proj_\Params(\mu^{(k,\ell)} + \kappa^j d^{(k,\ell)}) \in\Params,
\end{equation}
where $\kappa\in(0,1)$ and $d^{(k,\ell)}$ is the chosen descent direction at the iteration $(k,\ell)$. 
In our case, we make the standard choice
\begin{align*}
d^{(k,\ell)}=-(\mathcal R^{(k)}_\red(\mu^{(k,\ell)}))^{-1}\nabla_\mu \cJhatn^{(k)}(\mu^{(k,\ell)}) && \text{ for all } k,\ell\in \mathbb{N},\, \ell\geq 1,
\end{align*}
where 
\begin{align*}
\mathcal R^{(k)}_\red(\mu) = \left\{ \begin{array}{ll}
\delta_{ij} & \text{if } i\in \mathcal{A}^\varepsilon(\mu) \text{ or } j\in \mathcal{A}^\varepsilon(\mu)\\
(\cHhatn(\mu))_{i,j} & \text{otherwise},
\end{array} \right. && \text{for } \mu\in\Params.
\end{align*}
The function $\delta_{ij}$ indicates the Kronecker delta and the set $\mathcal{A}^\varepsilon$ is the $\varepsilon$-active set for the parameter constraints, i.e.
\[
\mathcal{A}^\varepsilon(\mu) = \left\{i\in\{1,\ldots,P\}\big| (\mu_\mathsf{b})_i-\mu_i \leq \varepsilon \text{ or } \mu_i-(\mu_\mathsf{a})_i\leq \varepsilon \right\}.
\]
For further details on the projected Newton method, the choice of $\varepsilon$ and its effect on convergence of the method, we refer to \cite[Section 5.5]{Kel99}.
Note that $\cHhatn(\mu)$ (and thus $\mathcal R^{(k)}_\red(\mu)$) might not be positive definite for every $\mu\in\Params$.
Therefore we use a truncated Conjugate Gradient (CG) method to compute $d^{(k,\ell)}$, where the CG terminates when a negative curvature condition criterium is triggered.
In such a way, we ensure that $d^{(k,\ell)}$ (resulting from the possible premature termination of the CG) is still a descent direction.
The truncated CG is explained in \cite[Algorithm 7.1]{NW06}. Moreover, we enforce an Armijo-type condition
\begin{subequations}
\label{Arm_and_TRcond}
\begin{equation}
\label{Armijo}\cJhatn^{(k)}(\mu^{(k,\ell)}(j)) - \cJhatn^{(k)}(\mu^{(k,\ell)}) \leq  -\frac{\kappa_{\mathsf{arm}}}{\kappa^j} \| \mu^{(k,\ell)}(j)-\mu^{(k,\ell)}\|^2_2,
\end{equation}
with $\kappa_{\mathsf{arm}}=10^{-4}$ and the additional TR constraint on $\cJhatn^{(k)}$ 
\begin{equation}
\label{TR_radius_condition} q^{(k)}(\mu^{(k,\ell)}(j)):= \frac{\Delta_{\Jhat}(\mu^{(k,\ell)}(j))}{\cJhatn^{(k)}(\mu^{(k,\ell)}(j))} \leq \delta^{(k)}
\end{equation}
\end{subequations}
by selecting $\mu^{(k,\ell+1)} = \mu^{(k,\ell)}(j^{(k,\ell)})$ for $\ell\geq 1$,
where $j^{(k,\ell)}<\infty$ is the smallest index for which \eqref{Arm_and_TRcond} holds.
From \cite{KMOSV20,QGVW2017}, we recall that the optimization sub-problem will terminate if
\begin{subequations}\label{Termination_crit_subproblem}
\begin{equation}
\label{FOC_subproblem}
\big\|\mu^{(k,\ell)}-\Proj_\Params(\mu^{(k,\ell)}-\nabla_\mu \Jhat_\red^{(k)}(\mu^{(k,\ell)}))\big\|_2\leq \tau_\text{\rm{sub}}
\end{equation}
or
\begin{equation}
\label{Cut_of_TR}
\beta_2\delta^{(k)} \leq \frac{\Delta_{\Jhat_\red^{(k)}}(\mu)}{\Jhat_\red^{(k)}(\mu)} \leq \delta^{(k)},
\end{equation}
\end{subequations}
\noindent where $\tau_\text{\rm{sub}}\in(0,1)$ is a predefined tolerance and $\beta_2\in(0,1)$, generally close to one. With condition \eqref{Cut_of_TR}, we prevent the sub-problem to spend too much time close to the boundary of the trust-region, because the model is poor in approximation; cf.~\cite{QGVW2017}. We also report the definition of AGC point for the constrained case.
\begin{definition}[AGC point for simple bounds]
	\label{Def:AGC}
	At the iteration $k$, we define the AGC point as
	\[
	\mu_\text{AGC}^{(k)}:= \mu^{(k,0)}(j^{(k)}_c)=  \Proj_\mathcal{P}(\mu^{(k,0)} + \kappa^{j^{(k)}_c} d^{(k,0)}),
	\]
	where $\mu^{(k,0)}:= \mu^{(k)}$, $d^{(k,0)}:= -\nabla_\mu \cJhatn^{(k)}(\mu^{(k,0)})$ and $j^{(k)}_c$ is the smallest non-negative integer $j$ for which $\mu^{(k,0)}(j)$ satisfies \eqref{Arm_and_TRcond} for $\ell=0$.
\end{definition}
Analogously to \cite{KMOSV20}, as an improvement over \cite{QGVW2017}, we also use a condition to enlarge the TR radius adaptively, which can significantly speed up the TR-RB method. To be more precise, we check whether the sufficient reduction predicted by the model function $\cJhatn^{(k)}$ is realized by the objective function, i.e.~
\begin{equation}
\label{TR_act_decrease}
\varrho^{(k)}:= \frac{\Jhat_h(\mu^{(k)})-\Jhat_h(\mu^{(k+1)})}{\cJhatn^{(k)}(\mu^{(k)})-\cJhatn^{(k)}(\mu^{(k+1)})}  \geq  \eta_\varrho
\end{equation}
for a tolerance $\eta_\varrho \in [3/4,1)$. Note that $\Jhat_h$ is available, after the
enrichment of the RB space \cite{KMOSV20}. In addition, since the dual solution of \eqref{eq:dual_solution} is included as snapshot, also the FOM gradient $\nabla_\mu \Jhat_h(\mu^{(k+1)})$ is available at this stage. Thus, we use it for computing the first-order critical condition for the outer TR method and hence to terminate the TR-RB algorithm. 
Notice that the choice of the (hidden) sub-problem solver differs from the one in \cite{KMOSV20,QGVW2017}, which requires the computation of the AGC point in advance, since it is not carried out naturally by the projected Newton method. Although this issue seems disadvantageous with respect to the projected BFGS method, where this computation is normally included in the process (cf.~\cite{KMOSV20,QGVW2017}), we remark that the search of the AGC point costs only one projected gradient optimization step and it is used as warm start for the projected Newton method. Therefore, the initial cost is justified by the subsequent advantage of the faster local quadratic convergence of the projected Newton method. It constitutes an improvement with respect to the projected BFGS method, in particular when the optimum is close to the boundary of the parameter set; cf.~\cite{Kel99,NW06}. 

Finally, we introduce the possibility of skipping to enrich the model if suitable conditions are satisfied. These conditions can be also used to accept the point $\mu^{(k+1)}$, since they directly imply the error-aware sufficient decrease condition \eqref{Suff_decrease_condition} for the convergence of the method; cf.~\cite{YM2013} and Section~\ref{sec:TR_convergence_analysis}. At first, we define
\begin{equation}
\label{eq:FOC_def}
g_h(\mu)  := \|\mu-\Proj_\Params(\mu-\nabla_\mu\Jhat_h(\mu))\|_2
\end{equation}
and analogously 
\[
g^{(k)}_\red(\mu) := \|\mu-\Proj_\Params(\mu-\nabla_\mu\cJhatn^{(k)}(\mu))\|_2
\]
for all $\mu\in\Params$. Then the sufficient condition for skipping the enrichment at iteration $k$ reads as follows:
\begin{equation}
\label{eq:skip_enrichment_condition}
\begin{aligned}
&&\textsf{Skip\_enrichment\_flag}(k) := \left(q^{(k)}(\mu^{(k+1)})\leq \beta_3\delta^{(k+1)}\right)\texttt{ and } \hspace{13mm} \\ &&  \left(\frac{\left|g_h(\mu^{(k+1)})-g_\red^{(k)}(\mu^{(k+1)})\right|}{g_\red^{(k)}(\mu^{(k+1)})}\leq \tau_g\right) \texttt{ and } \\ && \left(\frac{\|\nabla_\mu \Jhat_h(\mu^{(k+1)})-\nabla_\mu \Jhat^{(k)}_\red(\mu^{(k+1)})\|_2}{\|\nabla_\mu \Jhat_h(\mu^{(k+1)})\|_2 } \leq \min\{\tau_\text{\rm{grad}},\beta_3\delta^{(k+1)}\}\right)
\end{aligned}
\end{equation} 
for given $\tau_g>0$ and $\tau_\text{\rm{grad}},\beta_3\in(0,1)$. The first part of \eqref{eq:skip_enrichment_condition} indicates how much the current RB model is trustworthy in the next iteration $k+1$, the second condition is to ensure the convergence of the algorithm (cf.~Theorem~\ref{Thm:convergence_of_TR}) and the third one is to measure the RB accuracy in reconstructing the FOM gradient of $\Jhat_h$. Note that these conditions require FOM quantities. In \cite{KMOSV20}, they are accessible exactly because of the enrichment, therefore it appears contradictory to request them and then skip a basis update. Here the focus is in fact not to avoid particular FOM solves a-priori, but to exploit them in order to keep the dimension of the RB space small.
This is of particular importance when the TR-RB method takes many iterations (as seen in some examples in \cite{KMOSV20}), where an unconditional enrichment in each iteration would lead to overfitted and too large RB spaces, slowing down the computation in the long run.
\begin{algorithm2e}
	Initialize the ROM at $\mu^{(0)}$, set $k=0$ and \textsf{Loop\_flag}$=$\textsf{True}\; 
	\While{
		{\normalfont\textsf{Loop\_flag}}}{
		 Compute the AGC point $\mu^{(k)}_\text{AGC}$\;
		 Compute $\mu^{(k+1)}$ as solution of \eqref{TRRBsubprob} with stopping criteria \eqref{Termination_crit_subproblem}\label{TRRB-optstep}\;
		\uIf{\label{Suff_condition_TRRB}$\cJhatn^{(k)}(\mu^{(k+1)})+\Delta_{\Jhat_\red^{(k)}}(\mu^{(k+1)})<\cJhatn^{(k)}(\mu^{(k)}_\text{\rm{AGC}})$} {
			 Accept $\mu^{(k+1)}$, set $\delta^{(k+1)}=\delta^{(k)}$, compute $\varrho^{(k)}$ and $g_h(\mu^{(k+1)})$\label{AcceptingIterateSufficientCondition}\;
				\eIf { $g_h(\mu^{(k+1)}) \leq \tau_{\text{\rm{FOC}}}$ } {
					Set \textsf{Loop\_flag}$=$\textsf{False}\;
				} 
				{
					\If {$\varrho^{(k)}\geq\eta_\varrho$}  {
						Enlarge the TR radius $\delta^{(k+1)} = \beta_1^{-1}\delta^{(k)}$\;	
					}
					\If {{\normalfont\textbf{not}} {\normalfont\textsf{Skip\_enrichment\_flag}$(k)$}}
					{
						Update the RB model at $\mu^{(k+1)}$ 
						\label{UpdateRBModelSufficientCondition}\;
					}
				}
		}
		\uElseIf {\label{Nec_condition_TRRB}$\cJhatn^{(k)}(\mu^{(k+1)})-\Delta_{\Jhat_\red^{(k)}}(\mu^{(k+1)})> \cJhatn^{(k)}(\mu^{(k)}_\text{\rm{AGC}})$} {
			\If { $\beta_1\delta^{(k)} \leq \delta_{\text{\rm{min}}}$ \textbf{ or } \normalfont\textsf{Skip\_enrichment\_flag}$(k-1)$ \label{forced_enrichment}} {
				Update the RB model at $\mu^{(k+1)}$\; 
			}
			Reject $\mu^{(k+1)}$, shrink the radius $\delta^{(k+1)} = \beta_1\delta^{(k)}$ and go to \ref{TRRB-optstep}\; 
		}
		\Else {
			Compute $\Jhat_h(\mu^{(k+1)})$, $g_h(\mu^{(k+1)})$,
			$\varrho^{(k)}$
			  and set $\delta^{(k+1)} = \beta_1^{-1}\delta^{(k)}$\; 
			\eIf { $g_h(\mu^{(k+1)})\leq \tau_{\text{\rm{FOC}}}$ } {
				Set \textsf{Loop\_flag}$=$\textsf{False}\;
			} {
			\uIf { \label{AcceptingIterateNoUpdateByExactComputation_condition}{\normalfont\textsf{Skip\_enrichment\_flag}$(k)$} \textbf{ and } $\varrho^{(k)}\geq \eta_\varrho$}
			{ Accept $\mu^{(k+1)}$\label{AcceptingIterateNoUpdateByExactComputation}\;}
			\uElseIf {$\Jhat_h(\mu^{(k+1)}) \leq \cJhatn^{(k)}(\mu^{(k)}_\text{\rm{AGC}})$} {
				Accept $\mu^{(k+1)}$ and update the RB model
				\label{AcceptingIterateAndUpdateByExactComputation}\;	
				\If {$\varrho^{(k)}<\eta_\varrho$} {
				Set $\delta^{(k+1)}=\delta^{(k)}$\;
				}		
			}
			\Else{
				\If { $\beta_1\delta^{(k)} \leq \delta_{\text{\rm{min}}}$ \textbf{ or } \normalfont\textsf{Skip\_enrichment\_flag}$(k-1)$ \label{forced_enrichment_2}} {
					Update the RB model at $\mu^{(k+1)}$\; 
				}
				Reject $\mu^{(k+1)}$, set $\delta^{(k+1)} = \beta_1\delta^{(k)}$ and go to \ref{TRRB-optstep}\;		
			}
		}
		}
		Set $k=k+1$\;		
	}
  \caption{\footnotesize{TR-RB algorithm}}
	\label{Alg:TR-RBmethod}
\end{algorithm2e}

\subsection{Convergence result for the improved method}
\label{sec:TR_convergence_analysis}
In this section, we improve the convergence analysis done in \cite{KMOSV20}, first stating required assumptions from \cite{KMOSV20}. Condition \eqref{TR_radius_condition} imposes the following assumption to guarantee the well-posedness of the TR-RB algorithm.
\begin{assumption}
\label{asmpt:bound_J}
The cost functional $\J(u,\mu)$ is strictly positive for all $u\in V$ and all parameters $\mu\in\Params$. 
\end{assumption}
As also remarked in \cite{KMOSV20}, this assumption is not too restrictive.
Moreover, as pointed out in \cite{KMOSV20,QGVW2017,YM2013}, it is necessary that an error-aware sufficient decrease condition,
\begin{align}
\label{Suff_decrease_condition}
\cJhatn^{(k+1)}(\mu^{(k+1)})\leq \cJhatn^{(k)}(\mu_\text{\rm{AGC}}^{(k)}) && \text{ for all } k\in\mathbb{N},
\end{align}
is fulfilled at each iteration $k$ of the TR-RB algorithm. Cheaply computable sufficient and necessary conditions for \eqref{Suff_decrease_condition} in Algorithm~\ref{Alg:TR-RBmethod} (Step~\ref{Suff_condition_TRRB} and Step~\ref{Nec_condition_TRRB},
respectively) are considered to guarantee \eqref{Suff_decrease_condition}. The TR-RB algorithm rejects, then, any computed point which does not satisfy \eqref{Suff_decrease_condition}.
Algorithm~\ref{Alg:TR-RBmethod} may be trapped in an infinite loop, where every computed point is rejected and the TR radius is shrunk all time. This situation will not lead to convergence. We point out that this never happened in our numerical tests. On one hand, we consider two safe guards: the first is to force an update of the RB model when the TR radius is below a predefined threshold $0<\delta_{\text{\rm{min}}}\ll 1$ and a safety termination criteria, which is triggered when the TR radius is smaller than the double machine precision $\tau_{\rm{\text{mac}}}$. On the other hand, for showing convergence of Algorithm~\ref{Alg:TR-RBmethod}, we assume that this can not happen.
\begin{assumption}
\label{asmpt:rejection}
For each $k\geq 0$, there exists a radius $\delta^{(k)}>\tau_{\rm{\text{mac}}}>0$ for which there is a solution of
\eqref{TRRBsubprob} satisfying $\eqref{Suff_decrease_condition}$.
\end{assumption}
Another issue that might appear due to skipping a RB basis update is that at iteration $k-1$ the optimization subproblem terminates for \eqref{Cut_of_TR}, the point is accepted, the enrichment is skipped and the radius is enlarged to $\delta^{(k)}=\beta_1^{-1}\delta^{(k-1)}$, but then at iteration $k$ the point $\mu^{(k+1)}$ is rejected, implying to shrink the radius to the old value $\delta^{(k-1)}$. If the model is not updated also at this step, we are solving again the same subproblem at the next iteration starting at a point which was already triggering \eqref{Cut_of_TR}, therefore our step would be to compute only the AGC point. Although the method will converge anyway, this situation might repeat several times before we escape this ``problematic'' region, resulting in a waste of computational time, which contrasts all the time gained by the possibility of not enriching. Therefore, we impose an enrichment of the RB model, when the radius is shrunk at iteration $k$ and we skipped the basis update at iteration $k-1$; cf.~Step~\ref{forced_enrichment} and Step~\ref{forced_enrichment_2} of Algorithm~\ref{Alg:TR-RBmethod}. 
Note that to improve the convergence results, we required an additional assumption with respect to \cite{KMOSV20} (namely locally Lipschitz-continuous second derivatives in Assumption~\ref{asmpt:differentiability}), which we also require from the ROM.
\begin{assumption}
\label{asmpt:Lip_cont}
The ROM gradient $\nabla_\mu \Jhat_\red^{(k)}$ is uniformly Lipschitz- \\continuous, i.e. there exists a constant $C_L>0$ independent of $k$ such that
\[
\|\nabla_\mu \Jhat_\red^{(k)}(\mu)-\nabla_\mu\Jhat_\red^{(k)}(\nu)\|_2\leq C_L\|\mu-\nu\|_2
\]
holds for all $\mu,\nu\in\Params$ and all $k\in\mathbb{N}$. Similarly, the ROM second derivatives of $\Jhat_\red^{(k)}$ are locally Lipschitz-continuous.
\end{assumption}
This assumption restricts the set of cost functionals, nevertheless it guarantees a locally faster convergence behavior for this class. Algorithm~\ref{Alg:TR-RBmethod} is anyway still applicable to the general class of quadratic functionals -- and also converges in this case. We remark that Assumption~\ref{asmpt:Lip_cont} is needed for proving convergence of the method in an infinite-dimensional perspective. In the particular case of the RB model function, it is possible to show that this is satisfied. The proof follows from the fact that the RB model will exactly approximate the cost functional $\Jhat_h$ after a finite number of updates. As a direct consequence of Assumption~{\rm\ref{asmpt:Lip_cont}}, we have the following result:
\begin{corollary}
\label{cor:bound_grad}
Let Assumption~{\rm\ref{asmpt:Lip_cont}} be satisfied. Then there exists a constant $C>0$ such that for any $k\in\mathbb{N}$ it holds
\[
\|\nabla_\mu \Jhat_\red^{(k)}(\mu^{(k)})\|_2\leq C.
\]
\end{corollary}
Furthermore, the following property of the projection operator $\Proj_\Params$ holds:
\begin{lemma} \label{lemma:PropertyOfPParams}
	Let $\mu\in\Params$ and $d \in \R^P$ be arbitrary. Then it holds
	\begin{equation} 
	\label{eq:PropertyOfPParams}
	\left\Vert \mu - \Proj_\Params \left( \mu - t d \right) \right\Vert_2 \geq t \left\Vert \mu - \Proj_\Params \left( \mu - d \right) \right\Vert_2
	\end{equation}
	for all $t \in [0,1]$.
\end{lemma}

\begin{proof}
	Let $\mu \in \Params$, $d \in \R^P$ and $t \in [0,1]$ be arbitrary. We prove the statement by showing that 
	\begin{equation} \label{eq:PropertyOfPParams:Componentwise}
	\left| \mu_i - \left( \Proj_\Params \left( \mu - t d \right) \right)_i \right| \geq t \left|  \mu_i - \left( \Proj_\Params \left( \mu - d \right) \right)_i \right|
	\end{equation}
	holds for all components $i= 1,\ldots,P$. Let $i \in \{1,\ldots,P\}$ be arbitrary. \\ \emph{Case (1)}: $\left(  \Proj_\Params \left( \mu - t d \right) \right)_i \in \{ (\mu_\mathsf{a})_i,(\mu_\mathsf{b})_i \}$. \\
	It clearly holds $\left(  \Proj_\Params \left( \mu - t d \right) \right)_i = \left(  \Proj_\Params \left( \mu - d \right) \right)_i$, so that
		\[
		\left| \mu_i - \left(  \Proj_\Params \left( \mu - t d \right) \right)_i \right| = \left| \mu - \left(  \Proj_\Params \left( \mu - d \right) \right)_i \right| \geq t \left| \mu - \left(  \Proj_\Params \left( \mu - d \right) \right)_i \right|.
		\]
	\emph{Case (2a)}: $\left(  \Proj_\Params \left( \mu - t d \right) \right)_i \in ((\mu_\mathsf{a})_i,(\mu_\mathsf{b})_i)$ and $\left(  \Proj_\Params \left( \mu - d \right) \right)_i \in ((\mu_\mathsf{a})_i,(\mu_\mathsf{b})_i)$. We can conclude
		\[
		\left| \mu_i - \left(  \Proj_\Params \left( \mu - t d \right) \right)_i \right| = t \left| d_i \right| = t \left| \mu_i - \left(  \Proj_\Params \left( \mu - d \right) \right)_i \right|,
		\]
		which is what we have to show. \\
		\emph{Case (2b)}: $\left(  \Proj_\Params \left( \mu - t d \right) \right)_i \in ((\mu_\mathsf{a})_i,(\mu_\mathsf{b})_i)$ and 
		$\left(  \Proj_\Params \left( \mu - d \right) \right)_i \kern-0.1em\in\kern-0.1em \{ (\mu_\mathsf{a})_i,(\mu_\mathsf{b})_i \}$.
		We define $\tilde{t} := \frac{\mu_i -(\mu_{\mathsf{a},\mathsf{b}})_i}{d_i}$.
		Note that $t<\tilde{t}\leq 1$. Then it holds 
		\begin{align*}
		\left| \mu_i - \left( \Proj_\Params \left( \mu - t d \right) \right)_i \right| & = t \left| d_i \right| = \frac{t}{\tilde{t}} \left| \mu_i - \left( \Proj_\Params \left( \mu - \tilde{t} d \right) \right)_i \right| \\
		& = \frac{t}{\tilde{t}} \left| \mu_i - \left( \Proj_\Params \left( \mu - d \right) \right)_i \right| \geq t \left| \mu_i - \left( \Proj_\Params \left( \mu - d \right) \right)_i \right|.
		\end{align*} 
	Thus, in all cases, for any component $i \in \{1,\ldots,P\}$ the inequality \eqref{eq:PropertyOfPParams:Componentwise} holds.
	Now it can be immediately concluded that \eqref{eq:PropertyOfPParams} holds as well.
\end{proof} 
The next lemma is needed to show convergence of the algorithm.
\begin{lemma}
  \label{lemma:BothRBErrorsSmaller2}
	For every iterate $\mu^{(k)}$ ($k \in \N$) of Algorithm~\ref{Alg:TR-RBmethod}, it holds
	\begin{equation}
	q^{(k)}(\mu^{(k)}) \leq \beta_3\delta^{(k)} \quad \text{ and } \quad \frac{\left| g_\red^{(k)}(\mu^{(k)}) - g_h(\mu^{(k)}) \right|}{g_\red^{(k)}(\mu^{(k)})} \leq \tau_g \label{eq:BothRBErrorsSmaller2}
	\end{equation}
\end{lemma}

\begin{proof}
	We show this statement by induction over $k \in \mathbb{N}$. For $k = 0$ we trivially have 
	\[
	q^{(0)}(\mu^{(0)}) = 0 \quad \text{ and } \quad \frac{\left| g_\red^{(0)}(\mu^{(0)}) - g_h(\mu^{(0)}) \right|}{g_\red^{(0)}(\mu^{(0)})} = 0,
	\]
	since the RB model was constructed at $\mu^{(0)}$. \\
	Now assume that \eqref{eq:BothRBErrorsSmaller2} is satisfied for all $1 \leq l \leq k$ for some $k \in \N$ and let $\mu^{(k+1)}$ be the new accepted iterate. 
	\begin{enumerate}
		\item $\mu^{(k+1)}$ is accepted in line~\ref{AcceptingIterateSufficientCondition}: \\
		Then the RB model is updated in line~\ref{UpdateRBModelSufficientCondition}, if
		\[
		\begin{aligned}
		& q^{(k)}(\mu^{(k+1)}) > \beta_3\delta^{(k+1)} \quad \text{ or } \quad \frac{\left| g_\red^{(k)}(\mu^{(k+1)}) - g_h(\mu^{(k+1)}) \right|}{g_\red^{(k)}(\mu^{(k+1)})} > \tau_g \\ & \text{ or } \left(\frac{\|\nabla_\mu \Jhat_h(\mu^{(k+1)})-\nabla_\mu \Jhat^{(k)}_\red(\mu^{(k+1)})\|_2}{\|\nabla_\mu \Jhat_h(\mu^{(k+1)})\|_2 }> \min\{\tau_\text{\rm{grad}},\beta_3\delta^{(k+1)}\}\right).
		\end{aligned}
		\]
		So, on one hand, if the RB model is not updated in line~\ref{UpdateRBModelSufficientCondition}, this implies that 
		\[
		q^{(k+1)}(\mu^{(k+1)}) = q^{(k)}(\mu^{(k+1)}) \leq \beta_3\delta^{(k+1)}
		\] 
		and
		\[
		\frac{\left| g_\red^{(k+1)}(\mu^{(k+1)}) - g_h(\mu^{(k+1)}) \right|}{g_\red^{(k+1)}(\mu^{(k+1)})} = \frac{\left| g_\red^{(k)}(\mu^{(k+1)}) - g_h(\mu^{(k+1)}) \right|}{g_\red^{(k)}(\mu^{(k+1)})} \leq \tau_g
		\]
		hold. On the other hand, if the RB model is updated, we have
		\[
		q^{(k+1)}(\mu^{(k+1)}) = 0 \quad \text{ and } \quad \frac{\left| g_\red^{(k+1)}(\mu^{(k+1)}) - g_h(\mu^{(k+1)}) \right|}{g_\red^{(k+1)}(\mu^{(k+1)})} = 0,
		\]
		so that the claim follows in both cases.
		\item $\mu^{(k+1)}$ is accepted in line~\ref{AcceptingIterateNoUpdateByExactComputation}: \\
		In this case the RB model is not updated. Thus, we can directly conclude from the previous if-condition in line~\ref{AcceptingIterateNoUpdateByExactComputation_condition} and the enlarged TR radius
		\[
		q^{(k+1)}(\mu^{(k+1)}) = q^{(k)}(\mu^{(k+1)}) \leq \beta_3 \beta_1^{-1} \delta^{(k)} = \beta_3\delta^{(k+1)}
		\] 
		and
		\[
		\frac{\left| g_\red^{(k+1)}(\mu^{(k+1)}) - g_h(\mu^{(k+1)}) \right|}{g_\red^{(k+1)}(\mu^{(k+1)})} = \frac{\left| g_\red^{(k)}(\mu^{(k+1)}) - g_h(\mu^{(k+1)}) \right|}{g_\red^{(k)}(\mu^{(k+1)})} \leq \tau_g.
		\]
		Hence, the claim holds also in this case.
		\item $\mu^{(k+1)}$ is accepted in line~\ref{AcceptingIterateAndUpdateByExactComputation}: \\
		Then the RB model is updated at $\mu^{(k+1)}$, so that we have 
		\[
		q^{(k+1)}(\mu^{(k+1)}) = 0 \quad \text{ and } \quad \frac{\left| g_\red^{(k+1)}(\mu^{(k+1)}) - g_h(\mu^{(k+1)}) \right|}{g_\red^{(k+1)}(\mu^{(k+1)})} = 0.
		\]
		Thus, the claim holds trivially.
		
	\end{enumerate}
	In total, we have shown the claim for every possible case, which concludes the proof.
\end{proof}
We continue the convergence analysis by showing a result about the AGC point $\mu_\text{\rm{AGC}}^{(k)}$. We recall the following results from \cite[Corollary 5.4.4]{Kel99}:
\begin{lemma}
\label{lemma:kelley}
For all $j,k\in\mathbb{N}$ and $\kappa\in(0,1)$, we have
\begin{align*}
\|\mu^{(k)}-\Proj_\Params(\mu^{(k)}-&\kappa^j\nabla_\mu\Jhat^{(k)}_\red(\mu^{(k)}))\|^2_2 \\ &\leq \kappa^j \nabla_\mu\Jhat^{(k)}_\red(\mu^{(k)})\cdot(\mu^{(k)}-\Proj_\Params(\mu^{(k)}-\kappa^j\nabla_\mu\Jhat^{(k)}_\red(\mu^{(k)})))
\end{align*}
\end{lemma}
To proceed, we assume that the error indicator $q^{(k)}$ in the TR condition \eqref{TR_radius_condition} is uniformly continuous.
\begin{assumption}
\label{asmpt:unif_cont_q}
The function $q^{(k)}:\Params\to \mathbb{R}$ defined in \eqref{TR_radius_condition} is uniformly continuous in $\Params$ uniformly in $k$, i.e.
\[
\forall \varepsilon>0: \exists\eta = \eta(\varepsilon)>0: \forall\mu,\nu\in\Params \quad \|\mu-\nu\|_2<\eta \Rightarrow \left|q^{(k)}(\mu)-q^{(k)}(\nu)\right|<\varepsilon.
\]
\end{assumption}
Also this last assumption is needed for proving convergence of the method in an infinite-dimensional perspective. In the particular case of the RB model, since $\Params$ is compact, one can apply the Heine-Cantor theorem \cite{Rudin1976} to show that $q^{(k)}$ is uniformly continuous for each $k\in\mathbb{N}$. Then the independence from $k$ follows from the fact that the RB model approximation is exact (after a sufficient number of enrichments) and $q^{(k)}=q^{(k+1)}$ when the enrichment is not performed. Finally, the next result gives a lower and upper bound for the line-search of the AGC point. This is important, because it shows that at each iteration $k$ the ACG point can be computed in a finite number of line-search steps.
\begin{theorem}
\label{thm:linesearch-agc}
Let Assumptions~{\rm\ref{asmpt:parameter_separable}}-{\rm\ref{asmpt:rejection}} be satisfied and let $\mu^{(k)}(j):= \Proj_\Params(\mu^{(k)}-\kappa^j\nabla_\mu\Jhat^{(k)}_\red(\mu^{(k)}))$ for $j\in\mathbb{N}$. Then we have that $\mu^{(k)}(j)$ satisfies \eqref{Arm_and_TRcond} for all
\begin{equation}
\label{j_low_bound}
j \geq \log_\kappa\left(\min\left\{\frac{2(1-\kappa_{\mathsf{arm}})}{C_L},\frac{\eta((1-\beta_3)\tau_{\rm{\text{mac}}})}{C}\right\}\right)
\end{equation}
where $\kappa\in(0,1)$ is the backtracking constant introduced in \eqref{eq:General_Opt_Step_point} and $C_L$, $C$ and $\eta$ are introduced in Assumption~{\rm\ref{asmpt:Lip_cont}}, Corollary~{\rm\ref{cor:bound_grad}} and Assumption~{\rm\ref{asmpt:unif_cont_q}}, respectively. Furthermore, for the step-length of the AGC points, it holds
\begin{equation}
\label{j_up_bound}
j^{(k)}_c \leq \log_\kappa\left(\min\left\{\frac{2(1-\kappa_{\mathsf{arm}})\kappa}{C_L},\frac{\eta((1-\beta_3)\tau_{\rm{\text{mac}}})\kappa}{C}\right\}\right).
\end{equation}
\end{theorem}
\begin{proof}
We need to prove only \eqref{j_low_bound}, since \eqref{j_up_bound} is a direct consequence of it. Let $j$ satisfying \eqref{j_low_bound} be arbitrary and consider $y:= \mu^{(k)}-\mu^{(k)}(j)$, then it holds
\begin{equation}
\label{1step_AGC_thm}
\begin{aligned}
\Jhat_\red^{(k)}(\mu^{(k)})-&\Jhat_\red^{(k)}(\mu^{(k)}(j)) = -\int_0^1 \frac{\mathrm d}{\mathrm d s} \Jhat_\red^{(k)}(\mu^{(k)}-sy)\,\mathrm{d} s \\ & = \int_0^1  \nabla_\mu\Jhat^{(k)}_\red(\mu^{(k)}-sy)\cdot y \,\mathrm{d} s \\
& =  \nabla_\mu\Jhat(\mu^{(k)})\cdot y +\int_0^1\left(\nabla_\mu\Jhat^{(k)}_\red(\mu^{(k)}-sy)-\nabla_\mu\Jhat(\mu^{(k)})\right)\cdot y \,\mathrm{d} s
\end{aligned}
\end{equation}
Now, the integral term can be estimated exploiting the Lipschitz continuity of $\nabla_\mu \Jhat_\red^{(k)}$ (cf.~Assumption~{\rm\ref{asmpt:Lip_cont}}) as follows:
\begin{equation}
\label{int_estimate_Lip}
\begin{split}
\bigg| \int_0^1 \Big(\nabla_\mu\Jhat^{(k)}_\red(\mu^{(k)}-&sy)-\nabla_\mu\Jhat(\mu^{(k)})\Big)\cdot y \,\mathrm{d} s \bigg| \\&\leq C_L\int_0^1 \|sy\|_2\|y\|_2\,\mathrm{d}s = \frac{C_L}{2}\|\mu^{(k)}(j)-\mu^{(k)}\|^2_2
\end{split}
\end{equation}
Multiplying \eqref{1step_AGC_thm} by $\kappa^j$ and using \eqref{int_estimate_Lip} together with Lemma~\ref{lemma:kelley}, we obtain
\[
\begin{aligned}
\kappa^j\Big(\Jhat_\red^{(k)}(\mu^{(k)})&-\Jhat_\red^{(k)}(\mu^{(k)}(j))\Big) \\ & \geq \kappa^j\nabla_\mu\Jhat_\red^{(k)}(\mu^{(k)})\cdot(\mu^{(k)}-\mu^{(k)}(j))- \frac{C_L\kappa^j}{2}\|\mu^{(k)}(j)-\mu^{(k)}\|^2_2 \\
& \geq \left(1-\frac{C_L\kappa^j}{2}\right)\|\mu^{(k)}(j)-\mu^{(k)}\|^2_2.
\end{aligned}
\]
Thus, we have
\[
\Jhat_\red^{(k)}(\mu^{(k)}(j))-\Jhat_\red^{(k)}(\mu^{(k)})\leq -\frac{1}{\kappa^j}\left(1-\frac{C_L\kappa^j}{2}\right)\|\mu^{(k)}(j)-\mu^{(k)}\|^2_2.
\]
Since $j$ satisfies \eqref{j_low_bound}, we have that $\kappa^j\leq \frac{2(1-\kappa_{\mathsf{arm}})}{C_L}$. Therefore, the Armijo-type condition \eqref{Armijo} is satisfied. It remains to show that \eqref{TR_radius_condition} holds as well. Note that
\[
\|\mu^{(k)}(j)-\mu^{(k)}\|_2\leq \|\kappa^j\nabla_\mu\Jhat_\red^{(k)}(\mu^{(k)})\|_2\leq \eta((1-\beta_3)\tau_{\rm{\text{mac}}})
\]
by the choice of $j$ and Corollary~\ref{cor:bound_grad}. Now, Assumption~\ref{asmpt:rejection}, Lemma~\ref{lemma:BothRBErrorsSmaller2} and Assumption~\ref{asmpt:unif_cont_q} imply that 
\[
\begin{aligned}
q^{(k)}(\mu^{(k)}(j))&\leq|q^{(k)}(\mu^{(k)}(j))-q^{(k)}(\mu^{(k)})|+q^{(k)}(\mu^{(k)}) \\ & < (1-\beta_3)\tau_{\rm{\text{mac}}} +\beta_3\delta^{(k)}\leq \delta^{(k)},
\end{aligned}\] 
which completes the proof.
\end{proof}
In the next step we show that Algorithm~\ref{Alg:TR-RBmethod} ensures that the error-aware sufficient decrease condition \eqref{Suff_decrease_condition} is satisfied for every successful iteration.
\begin{lemma}
	 \label{lemma:EASDCFulfilled}
	Let the iterate $\mu^{(k+1)}$ be accepted by Algorithm~{\rm\ref{Alg:TR-RBmethod}}. Then the error-aware sufficient decrease condition \eqref{Suff_decrease_condition} is satisfied.
\end{lemma}
\begin{proof}
When the RB model is updated, we can proceed as in \cite[Section~4.1]{QGVW2017}. If the model is not enriched we have to distinguish two cases:
\begin{enumerate}
	\item $\mu^{(k+1)}$ is accepted in line~\ref{AcceptingIterateSufficientCondition}
	and \textsf{Skip\_enrichment\_flag} is true, we have
	\[
	\cJhatn^{(k+1)}(\mu^{(k+1)}) =\cJhatn^{(k)}(\mu^{(k+1)})\kern-0.1em \leq\kern-0.1em \cJhatn^{(k)}(\mu^{(k+1)}) + \Delta_{\cJhatn^{(k)}}(\mu^{(k+1)}) \kern-0.1em < \kern-0.1em \cJhatn^{(k)}(\mu_\text{\rm{AGC}}^{(k)}).
	\]
	\item $\mu^{(k+1)}$ is accepted in line~\ref{AcceptingIterateNoUpdateByExactComputation}: \\
	Note that it always holds $\cJhatn^{(k)}(\mu^{(k+1)}) \leq \cJhatn^{(k)}(\mu_\text{\rm{AGC}}^{(k)})$, since the truncated CG projected Newton method for solving the TR subproblem \eqref{TRRBsubprob} is a descent method and the AGC point $\mu_\text{\rm{AGC}}^{(k)}$ is used as a warm start. Since the RB model is not updated, this implies
	\[
	\cJhatn^{(k+1)}(\mu^{(k+1)}) = \cJhatn^{(k)}(\mu^{(k+1)}) \leq \cJhatn^{(k)}(\mu_\text{\rm{AGC}}^{(k)}),
	\]
	so that \eqref{Suff_decrease_condition} is satisfied.
\end{enumerate}
Thus, whenever $\mu^{(k+1)}$ is accepted, regardless updating the RB model or not, the error-aware sufficient decrease condition \eqref{Suff_decrease_condition} is satisfied.
\end{proof}
We are now able prove our improved (w.r.t.~\cite{KMOSV20}) convergence results also taking into consideration the possibility of skipping RB model updates.
\begin{theorem}
\label{Thm:convergence_of_TR}
	Let the hypotheses of Theorem~{\rm{\ref{thm:linesearch-agc}}} be satisfied.
	Then every accumulation point $\bar\mu$ of the sequence $\{\mu^{(k)}\}_{k\in\mathbb{N}}\subset \Params$ generated by Algorithm~{\rm\ref{Alg:TR-RBmethod}}
	is an approximate first-order critical point for $\Jhat_h$, i.e., it holds
	\begin{equation}
	\label{First-order_critical_condition}
	\|\bar \mu-\Proj_\Params(\bar \mu-\nabla_\mu \Jhat_h(\bar \mu))\|_2 = 0.
	\end{equation}
\end{theorem}
\begin{proof}
Let $k\in\mathbb{N}$ be arbitrary. From Definition~\ref{Def:AGC}, \eqref{Armijo} and  \eqref{Suff_decrease_condition} due to Lemma \ref{lemma:EASDCFulfilled}, we have
\[
\begin{aligned}
\Jhat_\red^{(k)}(\mu^{(k)})-\Jhat_\red^{(k+1)}(\mu^{(k+1)}) & \geq \Jhat_\red^{(k)}(\mu^{(k)})-\Jhat_\red^{(k)}(\mu^{(k)}_\text{AGC}) \\
& \geq \frac{\kappa_{\mathsf{arm}}}{\kappa^{j^{(k)}_c}}\|\mu^{(k)}-\mu^{(k)}_\text{AGC}\|^2_2 \\
& = \frac{\kappa_{\mathsf{arm}}}{\kappa^{j^{(k)}_c}}\|\mu^{(k)}-\Proj_\Params(\mu^{(k)}-\kappa^{j^{(k)}_c}\nabla_\mu\Jhat_\red^{(k)}(\mu^{(k)}))\|^2_2
\end{aligned}
\]
By summing both sides of the previous inequality from $k=0$ to $K$, we obtain
\[
\begin{aligned}
\Jhat_\red^{(0)}(\mu^{(0)})-\Jhat_\red^{(K+1)}(\mu^{(K+1)}) \geq \\ \sum_{k=0}^K \frac{\kappa_{\mathsf{arm}}}{\kappa^{j^{(k)}_c}}\|\mu^{(k)}-&\Proj_\Params(\mu^{(k)}-\kappa^{j^{(k)}_c}\nabla_\mu\Jhat_\red^{(k)}(\mu^{(k)}))\|^2_2 \geq 0.
\end{aligned}
\]
For $K\to +\infty$ the term on the left-hand side is bounded from above, due to Assumption~\ref{asmpt:bound_J}. Thus
\[
\lim_{k\to\infty} \frac{\kappa_{\mathsf{arm}}}{\kappa^{j^{(k)}_c}}\|\mu^{(k)}-\Proj_\Params(\mu^{(k)}-\kappa^{j^{(k)}_c}\nabla_\mu\Jhat_\red^{(k)}(\mu^{(k)}))\|^2_2 = 0.
\]
From Theorem~\ref{thm:linesearch-agc}, we have that $\kappa^{j^{(k)}_c}\geq \min\left\{\frac{2(1-\kappa_{\mathsf{arm}})}{C_L},\frac{\eta((1-\beta_3)\tau_{\rm{\text{mac}}})}{C}\right\}:= \widetilde{\kappa}$ for all $k\in\mathbb{N}$. Furthermore, we also have that $\kappa^{j^{(k)}_c}\leq 1$ for all $k\in\mathbb{N}$, because $\kappa\in(0,1)$. Hence,
\begin{align*}
\lim_{k\to\infty} \kappa_{\mathsf{arm}}\|\mu^{(k)}-\Proj_\Params(\mu^{(k)}-\widetilde{\kappa}\nabla_\mu\Jhat_\red^{(k)}(\mu^{(k)}))\|^2_2 & \leq \\ \lim_{k\to\infty} \frac{\kappa_{\mathsf{arm}}}{\kappa^{j^{(k)}_c}}\|\mu^{(k)}-\Proj_\Params(\mu^{(k)}-\kappa^{j^{(k)}_c}&\nabla_\mu\Jhat_\red^{(k)}(\mu^{(k)}))\|^2_2 = 0
\end{align*}
which clearly implies
\[
\lim_{k\to\infty} \|\mu^{(k)}-\Proj_\Params(\mu^{(k)}-\widetilde{\kappa}\nabla_\mu\Jhat_\red^{(k)}(\mu^{(k)}))\|_2 = 0.
\]
Lemma~\ref{lemma:PropertyOfPParams} shows that
\begin{align*}
\left\| \mu^{(k)} - \Proj_\Params \left( \mu^{(k)} - \widetilde{\kappa} \nabla_\mu \cJhatn^{(k)}(\mu^{(k)}) \right) \right\|_2 \geq \hspace{3em}\\ \widetilde{\kappa} \left\| \mu^{(k)} - \Proj_\Params \left( \mu^{(k)} - \nabla_\mu \cJhatn^{(k)}(\mu^{(k)}) \right) \right\|_2
\end{align*}
holds for all $k \in \N$. Thus, we can conclude
\begin{align}
\label{lim_k_to_zero}
& \lim_{k \to \infty} \left\| \mu^{(k)} - \Proj_\Params \left( \mu^{(k)} - \nabla_\mu \cJhatn^{(k)}(\mu^{(k)}) \right) \right\|_2 = 0.
\end{align} 
By Lemma~\ref{lemma:BothRBErrorsSmaller2}
 we have
\begin{align*}
& \left| \left\| \mu^{(k)} - \Proj_\Params \left( \mu^{(k)} - \nabla_\mu \cJhatn^{(k)}(\mu^{(k)}) \right) \right\|_2 - \left\| \mu^{(k)} - \Proj_\Params \left( \mu^{(k)} - \nabla_\mu \Jhat_h(\mu^{(k)}) \right) \right\|_2 \right| \\
& \qquad \leq \left\| \mu^{(k)} - \Proj_\Params \left( \mu^{(k)} - \nabla_\mu \cJhatn^{(k)}(\mu^{(k)}) \right) \right\|_2 \tau_g \to 0 \quad \text{as } k \to \infty,
\end{align*}
which, together with \eqref{lim_k_to_zero}, implies 
\begin{align}
\label{lim_k_to_zero_part2}
\lim_{k \to \infty}  \left\| \mu^{(k)} - \Proj_\Params \left( \mu^{(k)} - \nabla_\mu \Jhat_h(\mu^{(k)}) \right) \right\|_2 = 0.
\end{align}
Now let $\bar{\mu}$ be an accumulation point of the sequence $\{\mu^{(k)}\}_{k\in\mathbb{N}}\subset \Params$, i.e., it holds 
\[
\mu^{(k_i)} \to \bar{\mu} \quad \text{ as } i \to \infty
\]
for some subsequence $\left\{\mu^{(k_i)}\right\}_{i \in \mathbb{N}}$. Using \eqref{lim_k_to_zero_part2}, we have by the continuity of the gradient $\nabla_\mu \Jhat_h$ and of the projection operator $\Proj_\Params$
\begin{align*}
\left\| \bar\mu - \Proj_\Params \left( \bar\mu - \nabla_\mu \Jhat_h(\bar\mu) \right) \right\|_2 \kern-0.1em=\kern-0.1em 
\lim_{i \to \infty} \left\| \mu^{(k_i)} - \Proj_\Params \left( \mu^{(k_i)} - \nabla_\mu \Jhat_h(\mu^{(k_i)}) \right) \right\|_2 = 0,
\end{align*}
which concludes the proof.
\end{proof}
\begin{remark}
	\label{rmk:improved_conv_results}
(1) The biggest improvement in comparison to the convergence proof from {\rm\cite{KMOSV20}} consists in the fact that we proved that any accumulation point of the sequence $\left\{\mu^{(k)}\right\}$ is an actual critical point for $\Jhat_h$ and not an approximated one up to the tolerance $\tau_\text{\rm{sub}}$. \\
(2) Note that there is no direct use of the RB model properties in Theorem~{\rm\ref{Thm:convergence_of_TR}}, as it is in {\rm\cite{KMOSV20}}. This opens the possibility of considering different model functions, similarly to {\rm\cite{YM2013}} for unconstrained parameter sets, provided that the requested (and shown) properties hold. \\
(3) In contrast to {\rm\cite{KMOSV20}}, we make use of the error-aware sufficient decrease condition {\rm\eqref{Suff_decrease_condition}} in the proof of convergence and not only to guarantee that the accumulation point is not a local maximum of $\Jhat_h$. As also remarked in {\rm\cite{KMOSV20}}, $\bar\mu$ can still be a saddle point as well as a local minimum. In the numerical experiments, to verify that the computed point $\bar \mu$ is actually a local minimum, we check the second-order sufficient optimality conditions (cf.~Proposition~{\rm\ref{prop:second_order}}) as soon as the algorithm terminates. \\
(4) Algorithm~\ref{Alg:TR-RBmethod} is also an improvement with respect to {\rm\cite[Algorithm~1]{KMOSV20}} by allowing to skip updates of the RB model. This prevents the dimension of the RB space to grow excessively and, thus, helps to overcome the dimension of the FOM model. This feature is particularly relevant in applications which require many iterations, such as PDE-constrained multiobjective optimization by scalarization methods {\rm\cite{BBV2017,Ehrgott2005,IUV2017}}. In there many optimization problems have to be solved iteratively, so that it is crucial for an efficient algorithm to keep the dimension of the RB space reasonably small. 
\end{remark}

\subsection{Construction of RB spaces}
\label{sec:construct_RB} 
For the construction of the required RB spaces $V^{\pr}_{\red}$, $V^{\du}_{\red}$ the numerical experiments in  \cite{KMOSV20} have shown that \emph{Lagrangian RB spaces} are more favorable compared to \emph{aggregated RB spaces} (i.e.~non separated primal and dual spaces), since the use of the used NCD-corrected cost functional fully overcomes the approximation error that comes with separating the RB spaces.
Thus, using aggregated RB spaces unnecessarily increases the size of the RB spaces which is particularly problematic when the number of iterations of Algorithm~\ref{Alg:TR-RBmethod} is large.
For this reason, we only focus on enrichment approaches where the NCD-corrected functional is required.
We state the Lagrangian RB spaces and introduce an additional enrichment strategy which is of particular interest for a projected Newton method.
After solving a sub-problem of the TR-algorithm \ref{Alg:TR-RBmethod}, we (optionally) enrich with the primal and dual
solutions $u_{h, \mu^{(k)}}, p_{h, \mu^{(k)}} \in V_h$, where $\mu^{(k)}\in\Params$ is the current iterate.
In addition, we potentially also have access to their respective sensitivities w.r.t a direction $\eta$,
i.e.~$d_{\eta} u_{h, \mu}, d_{\eta} p_{h, \mu} \in V_h$. 
For the projected Newton method, we also require directional derivatives of the primal and dual solutions.
However, it can not be guaranteed that a reduced solution $d_{\eta} u_{\red, \mu}$ is a good approximation of $d_{\eta} u_{h, \mu}$; cf.~\cite{KMOSV20}.
On the other hand, computing snapshots $d_{\mu_i} u_{h, \mu}$ component-wise (e.g.~for full Taylor RB spaces) results in a prohibitively large computational effort for a high dimension of the parameter space. Instead, we suggest to only add snapshots of a wisely-chosen direction $\eta$.
As discussed earlier, it is cheap to compute the gradient $\nabla_\mu \Jhat_h(\mu)$ if $u_{h, \mu}, p_{h, \mu} \in V_h$ are already available.
On top of that, we know that we need the directional sensitivities of $u_{h, \mu}$ and $p_{h, \mu}$ in the direction $\eta=\nabla_\mu \cJhatn(\mu)$ in order to proceed with the next sub-problem of the TR-RB algorithm.
Thus, we propose the following two enrichment strategies.  
\begin{enumerate}[(a)]
		\item \emph{Lagrangian RB spaces}: \label{enrich:lag} Following \cite{KMOSV20},
		we simply add each FOM solution to the corresponding RB space, i.e.~for a
		given $\mu \in \Params$, we enrich by
		$
		V^{\pr,k}_{\red} = V^{\pr,k-1}_{\red} \cup \{u_{h,\mu}\}, 
		V^{\du,k}_{\red} = V^{\du,k-1}_{\red} \cup \{p_{h,\mu}\}.
		$
	   \item \emph{Directional Taylor RB space} \label{enrich:dir_tay}
	   We compute a direction $\eta:=\nabla_\mu \cJhatn(\mu)$ from $u_{h, \mu}$ and $p_{h, \mu}$ and include the directional derivatives to the respective RB space, i.e.
		$
		V^{\pr,k}_{\red} = V^{\pr,k-1}_{\red} \cup \{u_{h,\mu}\} \cup \{d_{\eta} u_{h,\mu}\}, 
		V^{\du,k}_{\red} = V^{\du,k-1}_{\red} \cup \{p_{h,\mu}\} \cup \{d_{\eta} p_{h,\mu}\}.
		$	   
	\end{enumerate}
	For the sake of brevity, in Section~\ref{sec:num_experiments}, we present only the results for strategy \ref{enrich:dir_tay}. At \cite{code_2} 
	one can find the results also for the strategy \ref{enrich:lag}.

\section{Numerical experiments}
\label{sec:num_experiments} 
In this section, we show the numerical performance of the improved Algorithm~\ref{Alg:TR-RBmethod} in comparison to the one described in \cite{KMOSV20}. We further study how the a posteriori error estimate for the optimal parameter approximation can be used as a post-processing tool. The simulations have been performed with a \texttt{python} implementation, using pyMOR \cite{milk2016pymor} for the \texttt{numpy/scipy}-based discretization and the RB part. The code for this section can be found on \cite{code_2}, where we also provide \texttt{jupyter}-notebooks\footnote{See \url{https://github.com/TiKeil/Proj-Newton-NCD-corrected-TR-RB-for-pde-opt}.} with the results (and explanations) of the numerical tests presented below. 

\subsection{Computational details}
\label{sec:computational_details}
We define the discrete space $V_h$ as a piecewise linear FE space on a triangular mesh $\tau_h$, with a fine enough grid to fulfill Assumption \ref{asmpt:truth}. From a MOR point of view, the choice of the inner product $(\cdot \,,\cdot)_h$ has a large impact on the coercivity and continuity constants used in the error estimates in Section \ref{sec:a_post_error_estimates}. For this reason we define the inner product to be spanned by the energy product with respect to a fixed reference parameter $\check{\mu}$, i.e.~$(\cdot, \cdot)_h := a_{\check{\mu}}(\cdot \,,\cdot)$. For this product, it is easy to compute a lower bound for the coercivity constant of the bilinear form $\underline{\bformd}$ for all parameters in $\Params$. In particular, we can make use of the min-theta approach (see  \cite[Proposition 2.35]{HAA2017}). The continuity constants $\cont{\bformd}$, $\cont{\partial_{\mu_i} \bformd}$, $\cont{\kformd}$, $\cont{\partial_{\mu_i} \kformd}$,
 $\cont{\partial_{\mu_i} \jformd}$ and $\cont{\partial_{\mu_i} \lformd}$ can be computed using a max-theta approach
 (see \cite{HAA2017}).

 Another important technical detail is the preassembly of all high dimensional parts of the model in Section \ref{sec:rom} and the corresponding error estimators in Section~\ref{sec:a_post_error_estimates}, which can be carried out as usual with RB methods (see, e.g., \cite{HAA2017,HRS2016,QMN2016}).
For all experiments, we use an initial TR radius of $\delta_0 = 0.1$, a TR shrinking factor $\beta_1 = 0.5$, an Armijo step-length $\kappa = 0.5$, a safeguard for the TR boundary of $\beta_2 = 0.95$, a tolerance for enlarging the TR radius of
$\eta_\varrho = 0.75$, a stopping tolerance for the TR sub-problems of $\tau_{\text{\rm{sub}}} = 10^{-8}$, a maximum number of TR iteration
$K=60$, a maximum number of sub-problem iteration $K_{\text{\rm{sub}}} = 400$, a maximum number of Armijo iteration of $50$, $\varepsilon=10^{-8}$ for the $\varepsilon$-active set and optional enrichment parameters $\tau_g= \tau_{\text{\rm{FOC}}}/\tau_{\text{\rm{sub}}}$, $\beta_3=0.5$ and $\tau_\text{\rm{grad}}=0.01$. We also point out that the stopping tolerance for the FOC condition $\tau_{\text{\rm{FOC}}}$ is specified in each experiment. 
We compare four Algorithms:
\newline\noindent \textbf{FOM TR-Newton-CG \cite{NW06}}: following \cite[Algorithm~7.2]{NW06}, this method considers a standard FOM quadratic approximation for $\Jhat_h$ as model function and includes the computation of the Cauchy point as well as a way to handle the box constraints of $\Params$, following \cite[Section~16.7]{NW06}.
\newline\noindent \textbf{BFGS NCD TR-RB (UE) \cite{KMOSV20}}: this is the NCD-corrected method with BFGS sub-problem solver and Lagrange RBs following \cite[Algorithm~1]{KMOSV20} with unconditional enrichment (UE) of the RB spaces, where no reduced hessian or sensitivities of the primal and dual solutions are required.
\newline\noindent \textbf{Newton NCD TR-RB (UE)}: this is the NCD-corrected method from \cite{KMOSV20} with directional Taylor RB spaces, a projected Newton method for the TR sub-problems and unconditional enrichment (UE) of the RB spaces.
\newline\noindent \textbf{Newton NCD TR-RB (OE) [Alg.~\ref{Alg:TR-RBmethod}]}: this is the NCD-corrected method with directional Taylor RB spaces, a projected Newton method for the TR sub-problems and optional enrichment (OE) of the RB spaces from Algorithm~\ref{Alg:TR-RBmethod}.

\subsection{Model problem: Quadratic objective functional with elliptic PDE constraints}
\label{sec:model_problem}
 For our experiments we set the objective functional to be a weighted $L^2$-misfit on a domain of interest $D \subseteq \Omega$ and a weighted Tikhonov term, i.e.
 \begin{align*}
 \mathcal{J}(v, \mu) = \frac{\sigma_d}{2} \int_{D}^{} (v - u^{\text{d}})^2 + \frac{1}{2} \sum^{M}_{i=1} \sigma_i (\mu_i-\mu^{\text{d}}_i)^2 + 1,
 \end{align*} 
 with a desired state $u^{\text{d}}$ and desired parameter $\mu^{\text{d}}$.
 With respect to the formulation in \eqref{P.argmin}, we have
 \begin{align*}
 \Theta(\mu) &= \frac{\sigma_d}{2} \sum^{M}_{i=1} \sigma_i (\mu_i-\mu^{\text{d}}_i)^2 + \frac{\sigma_d}{2} \int_{D}^{} u^{\text{d}} u^{\text{d}}\\
 j_{\mu}(u) &= -\sigma_d \int_{D}^{} u^{\text{d}}u, \qquad \text{and} \qquad k_{\mu}(u,u) = \frac{\sigma_d}{2} \int_{D}^{} u^2
 \end{align*}
 Note that the formulation of $\mathcal{J}(v, \mu)$ is a very general choice. It is applicable to design optimization, optimal control as well as to inverse problems.
We remark that the constant term $1$ is added to fulfill Assumption~\ref{asmpt:bound_J} and does not influence the position of the local minima. As PDE-constraint, we consider the weak formulation of the parameterized equation
 \begin{equation} \label{eq:prot_state}
 \begin{split}
 -  \nabla \cdot \left( \kappa_{\mu}  \nabla u_{\mu} \right) &= f_{\mu} \hS{58} \text{in } \Omega, \\
 c_\mu ( \kappa_{\mu}  \nabla u_{\mu} \cdot n) &= (u_{\text{out}} - u_{\mu}) \hS{20} \text{on } \partial \Omega.
 \end{split}
 \end{equation}
 with parametric diffusion coefficient $\kappa_{\mu}$ and source $f_{\mu}$, outside temperature $u_{\text{out}}$ and robin function $c_\mu$. Accordingly, we have
 \begin{align*}
 a_{\mu}(v,w) &= \int_{\Omega}^{} \kappa_\mu  \nabla v \cdot  \nabla w \,\mathrm{d}x + \int_{\partial \Omega}^{} \frac{1}{c_\mu} v w \,\mathrm{d}S \\
 l_\mu(v) &= \int_{\Omega}^{} f_\mu v \,\mathrm{d}x +  \int_{\partial \Omega}^{} \frac{1}{c_\mu} u_{\text{out}} v \,\mathrm{d}S.
 \end{align*}
 and, in addition, we set the parameter box constraints
$
 \mu_i \in [\mu_i^{\text{min}}, \mu_i^{\text{max}}].
$

A possible application is the stationary distribution of heat in a building. Inspired by the blueprint of a building with windows, heaters, doors and walls, we consider a parameterized diffusion problem which is displayed in Figure \ref{ex1:blueprint}. We picked a certain domain of interest $D$ and we enumerated all windows, walls, doors and heaters separately. 
 \begin{figure}
 	\begin{subfigure}[b]{0.60\textwidth}
 		\includegraphics[width=\textwidth]{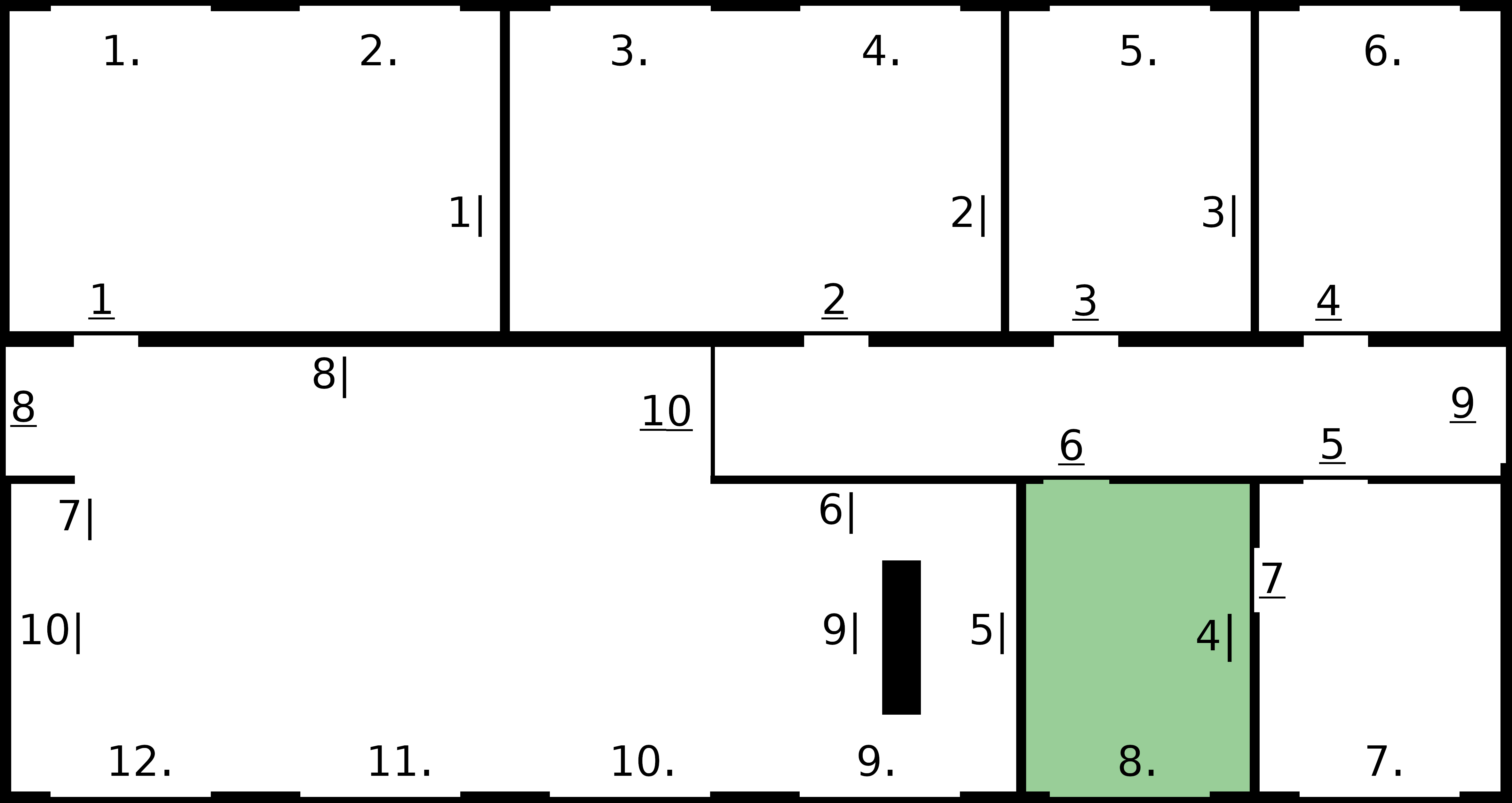}
 	\end{subfigure}
 	\centering
	\captionsetup{width=\textwidth}
 	\caption{\footnotesize{%
      Parameterization based on a blueprint of a building floor, see \cite{KMOSV20} for details. Numbers indicate potential parameters, where $1.$ is a window or a heater, $\underbar{1}$ are doors, and $1|$ are walls. The green region illustrates the domain of interest $D$.
}}
 	\label{ex1:blueprint}
 \end{figure}
 For simplicity we omit a realistic modeling of temperature and restrict ourselves to academic numbers of the diffusion and
 heat source quantities.
 We set the computational domain to $\Omega := [0,2] \times [0,1] \subset \mathbb{R}^2$ and we model boundary conditions by
 incorporating all walls and windows that touch the boundary of the blueprint to the robin function $c_\mu$. All other diffusion components enter the diffusion coefficient $\kappa_\mu$ whereas the heaters work as a source term on the right hand side $f_\mu$.
 Moreover, we assume an outside temperature of $u_{\text{out}}=5$.
 For our discretization we choose a mesh size $h= \sqrt{2}/200$ which resolves all features from the given picture and results
 in $\dim V_h = 80601$ degrees of freedom.\\
 We consider two scenarios.
 \begin{itemize}
   \item \emph{Experiment 1:} optimize 12 Parameters (3 walls, 2 doors, 7 heaters) to reach the target $u^\text{d}(x) = 18\rchi_{D}(x)$ and $\mu^d\equiv0$ using also the a posteriori error estimate for the optimal parameter (Proposition~\ref{apo-est-parameters}) as post-processing (cf. Section~\ref{num_test:12params}),
   \item \emph{Experiment 2:} optimize 28 Parameters (8 walls, 8 doors, 12 heaters) with target $u^\text{d}= \mathcal{S}_h(\mu^\text{d})$ where $\mu^\text{d}\in\Params$ is given
 	 (cf. Section~\ref{num_test:28params}). 
 \end{itemize}
Both experiments are computed with 10 different random samples for the starting parameter $\mu^{(0)}$.
For the sake of brevity, we omit the details on the data for the experiment and refer to \cite{code_2} on how to reproduce them.
\subsection{Experiment 1: A posteriori error estimate for optimal parameter}
\label{num_test:12params}
This experiment shows the usability of the a posteriori error estimate \eqref{apo-est-parameters} and shows the limitations of the projected BFGS method.
We focus, at first, on the behavior of the methods for a given starting parameter $\mu^{(0)}$.
In Fig.~\ref{Fig:EXC12:mu_error} the error at each iteration $k$ is reported for Algorithm~\ref{Alg:TR-RBmethod} and for the TR-RB method from \cite{KMOSV20}.
We omit the FOM TR-Newton-CG, due to its larger computational time.
We compute the solution with a tolerance $\tau_{\text{\rm{FOC}}}= 5\times 10^{-4}$ (Fig.~\ref{Fig:EXC12:mu_error}.A).
When the methods reach this tolerance, we evaluate the a posteriori error estimate and if this is greater
than the value $\tau_{\mu}= 10^{-4}$,
we decrease the tolerance $\tau_{\text{\rm{FOC}}}$ by two orders of magnitude and repeat the procedure
until the a posteriori estimate is below the desired tolerance $\tau_{\mu}$ (Fig.~\ref{Fig:EXC12:mu_error}.A).
In Fig.~\ref{Fig:EXC12:mu_error}.B, we do not use the a posteriori estimate and directly compute the solution with a tolerance $\tau_{\text{\rm{FOC}}}= 10^{-7}$.
\begin{figure}
  \footnotesize
\begin{tikzpicture}
\definecolor{color0}{rgb}{0.65,0,0.15}
\definecolor{color1}{rgb}{0.84,0.19,0.15}
\definecolor{color2}{rgb}{0.96,0.43,0.26}
\definecolor{color3}{rgb}{0.99,0.68,0.38}
\definecolor{color4}{rgb}{1,0.88,0.56}
\definecolor{color5}{rgb}{0.67,0.85,0.91}
\begin{axis}[
  name=left,
  anchor=west,
  scale=0.75,
log basis y={10},
tick align=outside,
tick pos=left,
x grid style={white!69.0196078431373!black},
xlabel={time in seconds [s]},
xmajorgrids,
xmin=0, xmax=265,
xtick style={color=black},
y grid style={white!69.0196078431373!black},
ymajorgrids,
ymin=2e-07, ymax=200,
ymode=log,
ylabel={\(\displaystyle \| \overline{\mu}_h-\mu^{(k)} \|_2\)},
yticklabels={,,},
ytick pos=right,
yticklabel pos=left,
ytick style={color=black}
]
\addplot [semithick, color0, mark=triangle*, mark size=3, mark options={solid, fill opacity=0.5}]
table {%
10.7283747196198 71.4005025909481
16.8894658088684 70.0913980253683
23.6503021717072 68.6311266604651
30.5194776058197 68.6184250989761
38.0436983108521 65.0536209016139
46.4489510059357 32.1153217100059
55.2568151950836 27.0201266807269
64.8881371021271 15.1115679924662
80.2292532920837 2.69350898011492
107.479714155197 0.0120659897452426
};
\addplot [semithick, color1, mark=*, mark size=3, mark options={solid, fill opacity=0.5}]
table {%
10.7799642086029 71.4005025909481
22.0994827747345 65.1324982951888
34.5064990520477 53.9934432453033
48.954788684845 31.6428672462588
63.0492160320282 0.22019511096089
};
\addplot [semithick, color2, mark=square*, mark size=3, mark options={solid, fill opacity=0.5}]
table {%
11.0857253074646 71.4005025909481
22.2940537929535 65.1324982951888
35.2455587387085 51.7977039852713
48.9639213085175 23.3668082033239
53.176020860672 0.226218016009454
};
\addplot [thick, dotted, color0, mark=triangle, mark size=3, mark options={solid, semithick}]
table {%
10.7593002319336 71.4005025909481
16.7505412101746 70.0913980253683
23.3089096546173 68.6311266604651
29.9739465713501 68.6184250989761
37.3876357078552 65.0536209016139
45.9340124130249 32.1153217100059
54.7701907157898 27.0201266807269
64.6972739696503 15.1115679924662
79.9411873817444 2.69350898011492
106.865009307861 0.0120659897452426
};
\addplot [thick, dashed, black, mark=square, mark size=0, mark options={solid, semithick}]
table {%
	106.865009307861 0.0120659897452426
	158.72991752624478 0.0120659897452426
};
\addplot [thick, dotted, color0, mark=triangle, mark size=3, mark options={solid, semithick}]
table {%
158.72991752624478 0.0120659897452426
174.107760667801 0.000461260770184682
184.492898464203 0.000456269534405666
195.489494085312 0.000453628915537476
207.455162286758 0.000451930571291389
220.113807439804 0.000438574164874584
234.120990991592 0.000437426210881451
251.963258266449 0.000211268220083034
266.428707361221 0.000211268166754213
};
\addplot [thick, dotted, color1, mark=o, mark size=3, mark options={solid, semithick}]
table {%
10.7847511768341 71.4005025909481
22.4626441001892 65.1324982951888
34.9769263267517 53.9934432453033
49.4036056995392 31.6428672462588
63.5665023326874 0.22019511096089
};
\addplot [thick, dashed,, black, mark=o, mark size=0, mark options={solid, semithick}]
table {%
	63.5665023326874 0.22019511096089
	117.5665023326874 0.22019511096089
};
\addplot [thick, dotted, color0, mark=o, mark size=3, mark options={solid, semithick}]
table {%
117.5665023326874 0.22019511096089
129.270685434341 0.00317499944918528
};
\addplot [thick, dashed,, black, mark=o, mark size=0, mark options={solid, semithick}]
table {%
129.270685434341 0.00317499944918528
178.270685434341 0.00317499944918528
};
\addplot [thick, dotted, color0, mark=o, mark size=3, mark options={solid, semithick}]
table {%
178.270685434341 0.00317499944918528
195.078726291656 2.1194544828339e-05
};
\addplot [thick, dashed,, black, mark=o, mark size=0, mark options={solid, semithick}]
table {%
	195.078726291656 2.1194544828339e-05
	246.078726291656 2.1194544828339e-05
};
\addplot [thick, dotted, color0, mark=o, mark size=3, mark options={solid, semithick}]
table {%
	246.078726291656 2.1194544828339e-05
};
\addplot [thick, dotted, color2, mark=square, mark size=3, mark options={solid, semithick}]
table {%
10.7282385826111 71.4005025909481
22.0372779369354 65.1324982951888
34.8145577907562 51.7977039852713
48.7089738845825 23.3668082033239
53.176020860672 0.226218016009454 
};
\addplot [thick, dashed,, black, mark=square, mark size=0, mark options={solid, semithick}]
table {%
	53.176020860672 0.226218016009454
	105.33605003356934 0.226218016009454
};
\addplot [thick, dotted, color2, mark=square, mark size=3, mark options={solid, semithick}]
table {%
	105.33605003356934 0.226218016009454 
	116.6974101066592 0.00138082875085338 
};
\addplot [thick, dashed,, black, mark=square, mark size=0, mark options={solid, semithick}]
table {%
	116.6974101066592 0.00138082875085338
	169.05726623535185 0.00138082875085338
};
\addplot [thick, dotted, color2, mark=square, mark size=3, mark options={solid, semithick}]
table {%
	169.05726623535185 0.00138082875085338 
	179.08393168449405 8.66810548955347e-07 
};
\addplot [thick, dashed, black, mark=square, mark size=0, mark options={solid, semithick}]
table {%
	179.08393168449405 8.66810548955347e-07
	231.2080514431 8.66810548955347e-07
};
\addplot [thick, dotted, color2, mark=square, mark size=3, mark options={solid, semithick}]
table {%
	231.2080514431 8.66810548955347e-07
};
\end{axis}
\begin{axis}[
  name=right,
  anchor=west,
  at=(left.east),
  xshift=1.3cm,
  scale=0.75,
legend cell align={left},
legend style={nodes={scale=0.7}, fill opacity=0.8, draw opacity=1, text opacity=1, at=(left.north east), anchor=north east, xshift=-5.21cm, yshift=-0.94cm, draw=white!80!black},
log basis y={10},
tick align=outside,
tick pos=left,
x grid style={white!69.0196078431373!black},
xlabel={time in seconds [s]},
xmajorgrids,
xmin=0, xmax=218,
xtick style={color=black},
y grid style={white!69.0196078431373!black},
ymajorgrids,
ymin=2e-07, ymax=200,
ymode=log,
ytick pos=left,
ytick style={color=black}
]
\addplot [semithick, color0, mark=triangle*, mark size=3, mark options={solid, fill opacity=0.5}]
table {%
10.9332320690155 71.4005025909481
17.0462052822113 70.0913980253683
23.8179948329926 68.6311266604651
30.5724768638611 68.6184250989761
38.0406160354614 65.0536209016139
46.4146301746368 32.1153217100059
55.305762052536 27.0201266807269
65.3911876678467 15.1115679924662
81.0280475616455 2.69350898011492
108.373761177063 0.0120659897452426
126.480749845505 0.000461260770184682
136.940483808517 0.000456269534405666
147.813118219376 0.000453628915537476
160.001117944717 0.000451930571291389
172.585354804993 0.000438574164874584
186.605408668518 0.000437426210881451
204.703733205795 0.000211268220083034
219.356258392334 0.000211268166754213
};
\addlegendentry{BFGS NCD TR-RB (UE) \cite{KMOSV20}}
\addplot [semithick, color1, mark=*, mark size=3, mark options={solid, fill opacity=0.5}]
table {%
10.9879050254822 71.4005025909481
22.7747066020966 65.1324982951888
35.4346406459808 53.9934432453033
49.5354845523834 31.6428672462588
63.1189568042755 0.22019511096089
77.7369198799133 0.00317499944918528
93.7262444496155 2.1194544828339e-05
};
\addlegendentry{Newton NCD TR-RB (UE)}
\addplot [semithick, color2, mark=square*, mark size=3, mark options={solid, fill opacity=0.5}]
table {%
11.1055860519409 71.4005025909481
22.735223531723 65.1324982951888
35.5743370056152 51.7977039852713
49.1154987812042 23.3668082033239
62.9940128326416 0.226218016009454
77.7316734790802 0.00138082875085338
81.8664269447327 8.66810548955347e-07
};
\addlegendentry{Newton NCD TR-RB (OE) [Alg.~\ref{Alg:TR-RBmethod}]}
\end{axis}
\node[anchor=south, yshift=4pt] at (left.north) {(A) $\tau_{\text{\rm{FOC}}}= 5\times 10^{-4}$ + parameter control};
\node[anchor=south, yshift=4pt] at (right.north) {(B) $\tau_{\text{\rm{FOC}}}= 1\times 10^{-7}$};
\end{tikzpicture}
\captionsetup{width=\textwidth}
	\caption{\footnotesize{%
      Error decay and performance of selected algorithms defined in Section \ref{sec:computational_details} for experiment 1 from Section \ref{sec:model_problem} with unconditional enrichment (UE) vs.~optional enrichment (OE) for a single optimization run with random initial guess $\mu^{(0)}$ for two choices of $\tau_\text{FOC}$ (solid lines) with optional intermediate parameter control according to
      \eqref{apo-est-parameters} (dotted lines): for each algorithm each marker corresponds to one (outer) iteration of the optimization method and indicates the absolute error in the current parameter, measured against the computed FOM optimum. The dashed black horizontal lines indicate the time taken for the post-processing of the parameter control.
}}
	\label{Fig:EXC12:mu_error}
\end{figure}
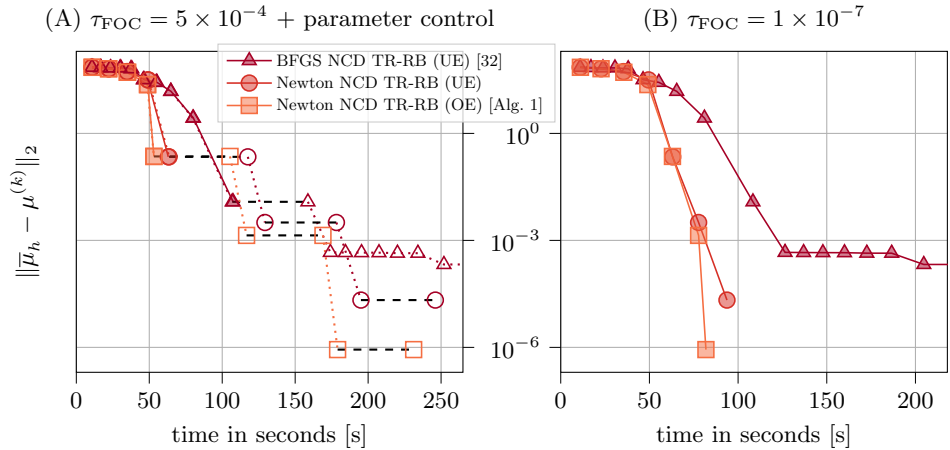
This test shows that the possibility of skipping enrichments improve the TR-RB algorithm. We point out that in this particular figure, the cost of computing the a posteriori error estimate \eqref{apo-est-parameters} is included in the computational time (as dashed horizontal line), which also includes
the costly computation of the smallest eigenvalue of the FOM hessian affecting the real performances of the method.
In fact, when directly considering a smaller $\tau_{\text{\rm{FOC}}}$, Algorithm~\ref{Alg:TR-RBmethod} is the fastest as visible from Fig.~\ref{Fig:EXC12:mu_error}.B.
This demonstrates how the a posteriori estimate is important as post-processing tool to verify the correct choice of the tolerances.
Another important issue that emerges from this numerical test is how a BFGS-based method (as the one in \cite{KMOSV20}) struggles to reach smaller values of the FOC condition, resulting in a high increase of the computational time and stagnating error, as can be seen in Figures~\ref{Fig:EXC12:mu_error}.A and \ref{Fig:EXC12:mu_error}.B. 
\begin{table}
  \footnotesize
	\centering
  \begin{tabular}{l|cc|c|cc}
    & runtime[s]
    & & iterations $k$\\
    & avg.~(min/max)
    & speed-up
    & avg.~(min/max)
    & rel.~error
    & FOC cond.\\\hline
    FOM   & 1381 (1190/1875) &  -- & 16.8 (14/23) & 2.66e-7 & 9.66e-8 \\
    BFGS (UE) & 818 (722/895) & 1.7 & 60 (60/60) & 2.77e-6 & 4.58e-7 \\
    Newton (UE) & 133 (94/212) & 10.4 & 7.8 (6/10) & 1.70e-7 & 1.73e-8 \\
    Newton (OE)  & 102 (82/141) & 13.6 &  6.9 (6/8) & 1.19e-7 & 1.30e-8
  \end{tabular}
  \captionsetup{width=\textwidth}
\caption{\footnotesize{%
    Performance and accuracy of the algorithms defined in Section \ref{sec:computational_details} (abbreviated in order of definition) for experiment 1 from Section \ref{sec:model_problem} with unconditional enrichment (UE) vs.~optional enrichment (OE) for ten optimization runs with random initial guess $\mu^{(0)}$ and $\tau_{\text{\rm{FOC}}}=10^{-7}$: averaged, minimum and maximum total computational time (column 2) and speed-up compared to the FOM variant (column 3); average, minimum and maximum number of iterations $k$ required until convergence (column 4), average relative error in the parameter (column 5) and average FOC condition (column 6).
    \label{Tab:EXC12}
}}
\end{table}
In Tab.~\ref{Tab:EXC12} we report the average computational time and iterations for ten random starting
parameters $\mu^{(0)}$ together with the relative error in reconstructing the local minimizer and the FOC
condition at which the method stops.
Also here, one can see how the possibility of skipping enrichments and the choice of a projected Newton method improve the results obtained with the algorithm from \cite{KMOSV20}.
In particular, the projected BFGS method struggles to reach the given $\tau_{\text{\rm{FOC}}}$ in all experiments, showing its limitation.
We remark that for larger tolerances the method from \cite{KMOSV20} is still valid and might converge faster, depending on the given example.
\subsection{Experiment 2: Large parameter set}
\label{num_test:28params}
In this experiment we apply the TR-RB algorithm to a 28 dimensional parameter set. This large number of parameters is prohibitive for the standard RB implementation based on a greedy algorithm for the offline phase.
This problem can be easily overcome by our adaptive algorithm, as also remarked in \cite{KMOSV20,QGVW2017}.
What might still be problematic is the increase of the number of iterations, which would lead to large RB spaces with the unconditional enrichment from \cite{KMOSV20}.
The purpose of this experiment is to demonstrate that skipping enrichment yields similar convergence behavior (in terms of iterations), but at a lower cost. 
\begin{figure}
  \footnotesize
\begin{tikzpicture}
\definecolor{color0}{rgb}{0.65,0,0.15}
\definecolor{color1}{rgb}{0.84,0.19,0.15}
\definecolor{color2}{rgb}{0.96,0.43,0.26}
\definecolor{color3}{rgb}{0.99,0.68,0.38}
\definecolor{color4}{rgb}{1,0.88,0.56}
\definecolor{color5}{rgb}{0.67,0.85,0.91}
\begin{axis}[
  name=left,
  anchor=west,
  width=6.5cm,
  height=4.5cm,
  log basis y={10},
  tick align=outside,
  tick pos=left,
  x grid style={white!69.0196078431373!black},
  xlabel={time in seconds [s]},
  xmajorgrids,
  xtick style={color=black},
  y grid style={white!69.0196078431373!black},
  ymajorgrids,
  ymode=log,
  ylabel={\(\displaystyle \| \overline{\mu}-\mu^{(k)} \|_2\)},
  ytick style={color=black}
]
\addplot [semithick, color0, mark=triangle*, mark size=3, mark options={solid, fill opacity=0.5}]
table {%
11.5934023857117 135.692657995198
19.2377111911774 90.8446588955084
27.8750627040863 24.6442190251569
37.0965969562531 9.75267373981067
47.1908657550812 4.48305564438341
58.3313579559326 2.84545730545265
70.7895290851593 1.68732573675647
83.6306612491608 1.45223566118595
97.9136147499084 0.631367397506257
113.420466423035 0.343342059768175
128.735586166382 0.307260609200393
144.771761417389 0.0897779774035193
167.995193243027 8.72687899118555e-05
190.073382377625 1.59054173003722e-06
210.727612018585 7.22008805445801e-07
};
\addplot [semithick, color1, mark=*, mark size=3, mark options={solid, fill opacity=0.5}]
table {%
11.9374821186066 135.692657995198
27.3678214550018 90.8498097955523
44.4709420204163 27.8277653035942
64.5053310394287 7.95877703598574
85.2987916469574 3.25287711111348
111.669808626175 2.16159970164513
139.39608001709 1.36285213889077
178.082042694092 9.39198337447456e-05
210.767727851868 2.56065163956946e-06
};
\addplot [semithick, color2, mark=square*, mark size=3, mark options={solid, fill opacity=0.5}]
table {%
11.6395268440247 135.692657995198
27.2452867031097 90.8498097955523
33.3784108161926 21.2028480910588
49.9773924350739 13.8959573913167
70.7744832038879 2.51586455562634
94.8881878852844 1.59200009037262
124.319049119949 0.00336062732991893
156.706687688828 0.000906237762678042
169.492668151855 2.72660621684241e-06
};
\end{axis}
\begin{axis}[
  name=right,
  anchor=west,
  at=(left.east),
  xshift=0.1cm,
  width=6.5cm,
  height=4.5cm,
  legend cell align={left},
  legend style={nodes={scale=0.7}, fill opacity=0.8, draw opacity=1, text opacity=1, at=(left.north east), anchor=north east, xshift=-2.0cm, yshift=1.0cm, draw=white!80!black},
  tick align=outside,
  tick pos=left,
  x grid style={white!69.0196078431373!black},
  xlabel={outer TR iteration $k$},
  ylabel={sub-problem iterations $L$},
  xmajorgrids,
  xtick style={color=black},
  y grid style={white!69.0196078431373!black},
  ymajorgrids,
  ytick pos=right,
]
\addplot [semithick, color0, mark=triangle*, mark size=3, mark options={solid, fill opacity=0.5}]
table {%
0 4
1 4
2 4
3 3
4 3
5 4
6 3
7 6
8 6
9 2
10 5
11 32
12 21
13 7
};
\addlegendentry{BFGS NCD TR-RB (UE) \cite{KMOSV20}}
\addplot [semithick, color1, mark=*, mark size=3, mark options={solid, fill opacity=0.5}]
table {%
0 4
1 4
2 4
3 3
4 5
5 4
6 11
7 8
};
\addlegendentry{Newton NCD TR-RB (UE)}
\addplot [semithick, color2, mark=square*, mark size=3, mark options={solid, fill opacity=0.5}]
table {%
0 4
1 3
2 2
3 4
4 5
5 11
6 11
};
\addlegendentry{Newton NCD TR-RB (OE) [Alg.~\ref{Alg:TR-RBmethod}]}
\end{axis}
\end{tikzpicture}
\captionsetup{width=\textwidth}
  \caption{\footnotesize{%
  Error decay w.r.t. the desired parameter $\bar\mu= \mu^\text{d}$ and performance (left) and number of sub-problem iterations in each TR iteration (right) of selected algorithms defined in Section \ref{sec:computational_details} for experiment 2 from Section \ref{sec:model_problem} with unconditional enrichment (UE) vs.~optional enrichment (OE) for a single optimization run with random initial guess $\mu^{(0)}$ for $\tau_\text{FOC} = 1\cdot10^{-5}$.}}
	\label{Fig:EXC28}
\end{figure}
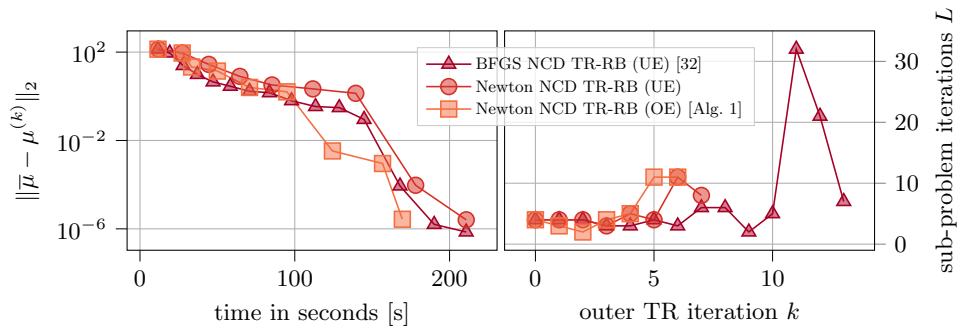
\begin{table}
  \footnotesize
	\centering
  \begin{tabular}{l|cc|c|cc}
    & runtime[s]
    & & iterations $k$\\
    & avg.~(min/max)
    & speed-up
    & avg.~(min/max)
    & rel.~error
    & FOC cond.\\\hline
    FOM   & 2423 (1962/3006) & -- & 18.5 (16/23) & 5.11e-9 & 4.57e-6 \\
    BFGS (UE) & 197 (156/272) & 12.3 & 12.1 (10/14) & 3.65e-9 & 2.38e-6 \\
    Newton (UE) & 258 (202/387) & 9.4 & 8.1 (7/9) & 5.83e-9 & 1.93e-6 \\
    Newton (OE)  & 168 (145/191) & 14.4 &  8.4 (7/12) & 1.22e-8 & 3.36e-6
  \end{tabular}
  \captionsetup{width=\textwidth}
  \caption{\footnotesize{%
      Performance and accuracy of the algorithms defined in Section \ref{sec:computational_details} for experiment 2 from Section \ref{sec:model_problem} for ten optimization runs with random initial guess $\mu^{(0)}$ and $\tau_{\text{\rm{FOC}}}=10^{-5}$, compare Table \ref{Tab:EXC12}.
  \label{Tab:EXC28}}}
\end{table}
Tab.~\ref{Tab:EXC28} reports the average runtime and iterations for the tested TR methods together with the relative error in reconstructing $\mu^\text{d}$ and at which FOC condition the method stops.
One can note that all the adaptive TR-RB algorithm are a valid tool for speeding up the computational time w.r.t.~the FOM TR-New.-CG.
Among all, the best performances are achieved by Algorithm~\ref{Alg:TR-RBmethod}.
The effect of skipping an enrichment can be seen comparing Algorithm~\ref{Alg:TR-RBmethod} with the method from \cite{KMOSV20} with the projected Newton method as sub-problem solver.
The numbers of outer iterations are the same, while the computational time is decreased.
This is due to two reasons: skipping an enrichment implies no preparation of the new RB space (like preassembling the new a posteriori estimate $\Delta_{\Jhat}$) and faster computations having a smaller RB space.
In Fig.~\ref{Fig:EXC28} (left), one can see the error between the desired parameter $\mu^\text{d}$ and each iteration of the different adaptive TR-RB methods for the same random starting parameter $\mu^{(0)}$, which confirms what is mentioned above.
Fig.~\ref{Fig:EXC28} (right) shows instead the number of iterations needed to solve each TR sub-problem at the outer iteration $k$ of the method.
One can deduce that the advantages of Algorithm~\ref{Alg:TR-RBmethod} with respect to \cite{KMOSV20} based on projected Newton are not due to different inner iterations number,
but have to be associated to the reduction of the dimension of the RB space.
From Fig.~\ref{Fig:EXC28} (right) we also see that the BFGS method might loose its super-linear convergence according to the approximation of the hessian carried out by the method, which might deteriorate for an increasing number of active components of the parameter $\mu^{(k)}$; see \cite{Kel99}.

\section{Conclusion}
In this work we proposed a new variant of adaptive TR-RB method for PDE-constrained parameter optimization. As major improvement we included the possibility of skipping basis updates according to rigorous criteria, which ensures the convergence of the method while preventing the RB space from growing indefinitely. We further made use of a projected Newton method for solving the TR sub-problem and of a post-processing operation based on an a posteriori error estimate for the optimal parameter. These features made the algorithm robust (in terms of convergence) and comparable with other methods presented in the literature, which performed slower than Algorithm~\ref{Alg:TR-RBmethod} in our numerical experiments. In future works, we are interested in adopting (spatially) localized RB methods to only locally enriching the RB model, allowing for faster computations and even smaller local dimensions of the ROM.

{\small
	\bibliographystyle{abbrv}

}

\section*{A Appendix}

\begin{propositionproof}[Upper bound on the model reduction error of the hessian of the reduced output]
	For the hessian $\HHhat_{h,\mu}(\mu)$ of $\Jhat_h(\mu)$ and the true hessian 
	$\cHhatn(\mu)$ of the NCD-corrected functional from Proposition {\rm\ref{prop:true_corrected_reduced_hessian}},
	we have the following a posteriori error bound
	\begin{align*}
	\big|\HHhat_{h,\mu}(\mu) - \cHhatn(\mu)&\big| \leq \Delta_{{\HH}}(\mu) := \Big\| \big(\Delta_{\HH_{i,l}}(\mu)\big)_{i,l} \Big\|_2
	\end{align*}
	with
	\begin{align*}
	\Delta_{{\HH}_{i,l}}&(\mu)\\ :=& \,\, \Delta_\pr(\mu)\Big(
	\cont{\partial_{\mu_i}\partial_{\mu_l} \jformd}
	+ 2\cont{\partial_{\mu_i}\partial_{\mu_l} \kformd} \|u_{\red, \mu}\|
	+\cont{\partial_{\mu_i}\partial_{\mu_l} \bformd}\|p_{\red, \mu}\|
	\\
	&\quad \quad \quad \quad + 2\cont{\partial_{\mu_i} \kformd}\|\dred{\mu_l}u_{\red, \mu}\|
	+ \cont{\partial_{\mu_i} \bformd}\|\dred{\mu_l}p_{\red, \mu}\|
	\Big)
	\\
	&+\Delta_{d_{\mu_l}\pr}(\mu)\Big(
	\cont{\partial_{\mu_i} \jformd} + 2\cont{\partial_{\mu_i} \kformd}\|u_{\red, \mu}\| +\cont{\partial_{\mu_i} \bformd} \|p_{\red, \mu}\|
	\Big)
	\\
	&+\Delta_\du(\mu)\Big(
	\cont{\partial_{\mu_i}\partial_{\mu_l} \lformd}
	+\cont{\partial_{\mu_i}\partial_{\mu_l} \bformd} \|u_{\red, \mu}\|
	+\cont{\partial_{\mu_i} \bformd} \|d_{\mu_l}u_{\red, \mu}\|
	\Big)
	\\
	&+\Delta_{d_{\mu_l}\du}(\mu)\Big(
	\cont{\partial_{\mu_i} \lformd} +\cont{\partial_{\mu_i} \bformd} \|u_{\red, \mu}\|
	\Big) +(\Delta_\pr)^2(\mu)\Big(
	\cont{\partial_{\mu_i}\partial_{\mu_l} \kformd}
	\Big)
	\\
	&+\Delta_\pr(\mu)\Delta_\du(\mu)\Big(
	\cont{\partial_{\mu_i}\partial_{\mu_l} \bformd}
	\Big)+\Delta_\pr(\mu)\Delta_{d_{\mu_l}\pr}(\mu)\Big(
	2\cont{\partial_{\mu_i} \kformd}
	\Big)
	\\
	&+\Delta_\pr(\mu)\Delta_{d_{\mu_l}\du}(\mu)\Big(
	\cont{\partial_{\mu_i} \bformd}
	\Big)+\Delta_{d_{\mu_l}\pr}(\mu)\Delta_\du(\mu)\Big(
	\cont{\partial_{\mu_i} \bformd}
	\Big) \\
	&+ \cont{\partial_{\mu_i}\bformd} \|d_{\mu_l} u_{\red,\mu}\|~\|w_{\red,\mu}\| 
	+ \cont{\partial_{\mu_i}\lformd} \|d_{\mu_l} w_{\red,\mu}\| +  
	\cont{\partial_{\mu_i}\bformd}  \|u_{\red,\mu}\|~\|d_{\mu_l} w_{\red,\mu}\| \\
	&+ 2 \cont{\partial_{\mu_i}\kformd} \|z_{\red,\mu}\|~\|d_{\mu_l} u_{\red,\mu}\| 
	+ \cont{\partial_{\mu_i}a_\mu} \|z_{\red,\mu}\|~\|d_{\mu_l} p_{\red,\mu}\| 
	+ \cont{\partial_{\mu_i}j_\mu} \|d_{\mu_l} z_{\red,\mu}\| \\
	&+ 2\cont{\partial_{\mu_i}k_\mu} \|d_{\mu_l} z_{\red,\mu}\|~\|u_{\red,\mu}\| 
	+ \cont{\partial_{\mu_i}a_\mu} \|d_{\mu_l} z_{\red,\mu}\|~\|p_{\red,\mu}\| \\
	&+ \cont{\partial_{\mu_i}\partial_{\mu_l}\lformd} \|w_{\red,\mu}\| + \cont{\partial_{\mu_i}\partial_{\mu_l}\bformd} \|w_{\red,\mu}\|~\|u_{\red,\mu}\| \\
	&+ \cont{\partial_{\mu_i}\partial_{\mu_l}\jformd}\|z_{\red,\mu}\| 
	+2\cont{\partial_{\mu_i}\partial_{\mu_l}\kformd} \|z_{\red,\mu}\|~\|u_{\red,\mu}\|
	+\cont{\partial_{\mu_i}\partial_{\mu_l}\bformd} \|z_{\red,\mu}\|~\|p_{\red,\mu}\|
	\end{align*}
	where $\|\cdot\|_2$ denotes the spectral norm for matrices. The norms of the auxiliary functions $\|w_{\red,\mu}\|$, $\|z_{\red,\mu}\|$ and the norms of their sensitivities $\|d_{\mu_l}w_{\red,\mu}\|$ and 
	$\|d_{\mu_l}z_{\red,\mu}\|$ can be estimated by
	\begin{enumerate}[(i)]
		\item $\|z_{\red, \mu}\| \leq \underline{\bformd}^{-1} \|r_\mu^\pr(u_{\red,\mu})\|$,
		\item $\|w_{\red, \mu}\| \leq \underline{\bformd}^{-1} \left( \|r_\mu^\du(u_{\red,\mu},p_{\red,\mu})\| 
		+ 2\cont{k_\mu} \|z_{\red, \mu}\| \right)$,
		\item 	$\|d_{\mu_l}z_{\red,\mu}\| \leq \underline{\bformd}^{-1} \left(\|\resd^{\pr,d_{\mu_i}}(u_{\red, \mu}, d_{\mu_i}u_{\red, \mu})\| + \cont{\partial_{\mu_l} \bformd}\|z_{\red,\mu}\| \right) $,
		\item 	$\|d_{\mu_l}w_{\red,\mu}\|  \leq 
		\underline{\bformd}^{-1}\Big(\|\resd^{\du,d_{\mu_i}}(u_{\red, \mu}, p_{\red, \mu}, d_{\mu_i}u_{\red, \mu}, d_{\mu_i}p_{\red, \mu}) \| 
		+ 2\cont{k_\mu}\|d_{\mu_l}z_{\red,\mu}\| \\
		+ 2\cont{\partial_\mu k_\mu}\|z_{\red,\mu}\| + \cont{\partial_\mu \bformd}\|w_{\red,\mu}\| \Big)$.
	\end{enumerate}
\end{propositionproof}
\begin{proof} 
	To prove the hessian estimate, we recall that for all $i,l$ we have
	\[
	\begin{aligned}
	\big(\cHhatn(&\mu))_{i, l}  = \partial_\mu \left(j_\mu(d_{\mu_l} u_{\red,\mu})+2k_\mu(u_{\red,\mu},d_{\mu_l} u_{\red,\mu})  - a_\mu(d_{\mu_l} u_{\red,\mu},p_{\red,\mu}+w_{\red,\mu}) \right.\\
	& + r^\pr_\mu(u_{\red,\mu})[d_\nu p_{\red,\mu}+d_{\mu_l} w_{\red,\mu}] - 2 k_\mu (z_{\red,\mu},d_{\mu_l} u_{\mu,\red}) \\
	& + a_\mu(z_{\red,\mu},d_{\mu_l} p_{\red,\mu}) - r^\du_\mu(u_{\red,\mu},p_{\red,\mu})[d_{\mu_l} z_{\red,\mu}]   \\
	& \left.+\partial_\mu (\J(u_{\red,\mu},\mu) + r^\pr_\mu(u_{\red,\mu})[p_{\red,\mu}+w_{\red,\mu}]
	- r_\mu^\du(u_{\red,\mu},p_{\red,\mu})[z_{\red,\mu}])\cdot e_l \right)\cdot e_i,
	\end{aligned}
	\]   
	and
	\begin{align*}
	\big(\HHhat_{h,\mu}(\mu)\big)_{i, l} \kern -0.1em
	=  \kern -0.1em \partial_\mu\Big(&\kern -0.1em\partial_u \J(u_{h,\mu}, \mu)[d_{\mu_l}u_{h,\mu}] \kern -0.2em+\kern -0.2em r_\mu^\pr(u_{h,\mu})[d_{\mu_l}p_{h,\mu}]\kern -0.2em -\kern -0.1em a_{\mu}(d_{\mu_l}u_{h,\mu}, p_{h,\mu}) \\
	&+\partial_\mu\big(\J(u_{h,\mu}, \mu)+ r_\mu^\pr(u_{h,\mu})[p_{h,\mu}]\big)\cdot e_l\Big)\cdot e_i.
	\end{align*}
	We get
	\begin{align*}
	\big|\big(&\HHhat_{h,\mu}(\mu) - \cHhatn(\mu)\big)_{i, l}\big| \\ &\leq \big| \partial_\mu\Big(\partial_u \J(u_{h,\mu}, \mu)[d_{\mu_l}u_{h,\mu}] - \partial_u \J(u_{\red,\mu}, \mu)[	d_{\mu_l}u_{\red, \mu}] + l_{\mu}(d_{\mu_l}p_{h,\mu}) 
	\\&\quad-  l_{\mu}(d_{\mu_l}p_{\red, \mu}) 
	- a_{\mu}(d_{\mu_l}u_{h,\mu}, p_{h,\mu}) + a_{\mu}(d_{\mu_l}u_{\red, \mu}, p_{\red,\mu}) 
	- a_{\mu}(u_{h,\mu}, d_{\mu_l}p_{h,\mu}) \\
	&\quad+ a_{\mu}(u_{\red,\mu}, d_{\mu_l}p_{\red, \mu}) 
	+ a_\mu(d_{\mu_l} u_{\red,\mu},w_{\red,\mu}) - r^\pr_\mu(u_{\red,\mu})[d_{\mu_l} w_{\red,\mu}]\\ 
	&\quad+ 2 k_\mu (z_{\red,\mu},d_{\mu_l} u_{\mu,\red}) 
	- a_\mu(z_{\red,\mu},d_{\mu_l} p_{\red,\mu}) + r^\du_\mu(u_{\red,\mu},p_{\red,\mu})[d_{\mu_l} z_{\red,\mu}]\\
	&\quad+\partial_\mu\big(\J(u_{h,\mu}, \mu) - \J(u_{\red,\mu}, \mu) 
	+ l_{\mu}(p_{h,\mu}) - l_{\mu}(p_{\red,\mu}) 
	- a_{\mu}(u_{h,\mu}, p_{h,\mu}) \\
	&\qquad\quad+  a_{\mu}(u_{\red,\mu}, p_{\red,\mu}) 
	+ r^\pr_\mu(u_{\red,\mu})[w_{\red,\mu}]- r_\mu^\du(u_{\red,\mu},p_{\red,\mu})[z_{\red,\mu}] \big)\cdot e_l\quad\Big)\cdot e_i \,\big|. 
	\end{align*}
	For the first terms we see that 
	\begin{align*}
	\partial_\mu\big(\partial_u &\J(u_{h,\mu}, \mu)[d_{\mu_l}u_{h,\mu}] - \partial_u \J(u_{\red,\mu}, \mu)[\dred{\mu_l}u_{\red, \mu}] \big)\cdot e_i 
	\\ & =\partial_\mu \big( j_{\mu}(d_{\mu_l} e_{h, \mu}^\pr) + 2 k_{\mu}(d_{\mu_l} e_{h, \mu}^\pr,u_{h,\mu}) + 2 k_{\mu}(\dred{\mu_l}u_{\red, \mu},e_{h, \mu}^\pr) \big)\cdot e_i. 
	\end{align*}
	Obviously, this equation still incorporates the norm of the FOM solution $u_{h,\mu}$. However, we can simply estimate these FOM quantities by $\|u_{h,\mu}\| \leq \Delta_\pr(\mu) + \|u_{\red,\mu}\|$. 
	For the second terms we have
	\begin{align*}
	\partial_\mu(l_{\mu}(d_{\mu_l}p_{h,\mu}) -  l_{\mu}(\dred{\mu_l}p_{\red, \mu}) )\cdot e_i =  \partial_\mu(l_{\mu}(d_{\mu_l} e_{h, \mu}^\du))\cdot e_i.
	\end{align*}
	The third terms can be determined by
	\begin{align*}
	- \partial_\mu\big( a_{\mu}(d_{\mu_l}u_{h,\mu}, p_{h,\mu}) -& a_{\mu}(\dred{\mu_l}u_{\red, \mu}, p_{\red,\mu})\big)\cdot e_i \\
	&= -  \partial_\mu\big( a_{\mu}(d_{\mu_l} e_{h, \mu}^\pr, p_{h,\mu}) - a_{\mu}(\dred{\mu_l}u_{\red, \mu}, e_{h,\mu}^\du)\big)\cdot e_i
	\end{align*}
	and similarly we get for the fourth term
	\begin{align*}
	- \partial_\mu\big( a_{\mu}(u_{h,\mu}, d_{\mu_l}p_{h,\mu}) &- a_{\mu}(u_{\red,\mu}, \dred{\mu_l}p_{\red, \mu})\big)\cdot e_i \\
	&= -  \partial_\mu\big( a_{\mu}(e_{h, \mu}^\pr, d_{\mu_l}p_{h,\mu}) - a_{\mu}(u_{\red,\mu},d_{\mu_l} e_{h, \mu}^\du)\big)\cdot e_i.
	\end{align*}
	With the same strategy as above, we have for the first part of the second derivatives that
	\begin{align*}
	\partial_\mu\Big(&\partial_\mu\big(\J(u_{h,\mu}, \mu) - \J(u_{\red,\mu}, \mu) + l_{\mu}(p_{h,\mu}) - l_{\mu}(p_{\red,\mu}) \\
	&\qquad
	- a_{\mu}(u_{h,\mu}, p_{h,\mu}) +  a_{\mu}(u_{\red,\mu}, p_{\red,\mu}) \big)\cdot e_l\Big)\cdot e_i  \\
	&\hS{40} 
	= \partial_\mu\Big(\partial_\mu\big(j_{\mu}(e_{h, \mu}^\pr) + 2 k_{\mu}(e_{h, \mu}^\pr,u_{h,\mu}) + 2 k_{\mu}(u_{\red, \mu}^\pr,e_{h, \mu}^\pr)  
	+ l_{\mu}(e_{h, \mu}^\du) \\
	&\hS{110}
	- (a_{\mu}(e_{h, \mu}^\pr, p_{h,\mu}) -  a_{\mu}(u_{\red,\mu}, e_{h, \mu}^\du) \big)\cdot e_l\Big)\cdot e_i.
	\end{align*}
	For the rest of the proof we simply use the Cauchy-Schwarz inequality for all terms and sum all pieces together to get $\Delta_{{\HH}_{i,l}}(\mu)$.  
	
	For a proof of (i) and (ii) we refer to \cite[Proposition 2.7]{KMOSV20}. For the first sensitivity estimation (iii) we use the equations \eqref{z_eq_sens} and \eqref{sens_res_pr} to get
	\begin{align*}
	\underline{\bformd} \|d_{\mu_l}&z_{\red,\mu}\|^2 \leq \bformd(d_{\mu_l}z_{\red,\mu},d_{\mu_l}z_{\red,\mu}) \\&= 
	-\partial_\mu(r_\mu^\pr(u_{\red,\mu})[d_{\mu_l}z_{\red,\mu}] + a_\mu(z_{\red,\mu},d_{\mu_l}z_{\red,\mu}))\cdot \nu + a_\mu(d_\nu u_{\red,\mu},d_{\mu_l}z_{\red,\mu}) \\
	&= -\resd^{\pr,d_{\mu_i}}(u_{\red, \mu}, d_{\mu_i}u_{\red, \mu})[d_{\mu_l}z_{\red,\mu}] - \partial_\mu \bformd(z_{\red,\mu},d_{\mu_l}z_{\red,\mu}) \cdot \nu \\
	&\leq \left(\|\resd^{\pr,d_{\mu_i}}(u_{\red, \mu}, d_{\mu_i}u_{\red, \mu})\| + \cont{\partial_{\mu_l} \bformd}\|z_{\red,\mu}\| \right) \|d_{\mu_l}z_{\red,\mu}\|.
	\end{align*}
	For (iv), we instead use \eqref{w_eq_sens} and \eqref{sens_res_du} and yield
	\begin{align*}
	\underline{\bformd} \|&d_{\mu_l}w_{\red,\mu}\|^2 \leq \bformd(d_{\mu_l}w_{\red,\mu},d_{\mu_l}w_{\red,\mu}) \\
	&= \partial_\mu (r_\mu^\du(u_{\red,\mu},p_{\red,\mu})[d_{\mu_l}w_{\red,\mu}]-2k_\mu(z_{\red,\mu},d_{\mu_l}w_{\red,\mu})-a_\mu(d_{\mu_l}w_{\red,\mu},w_{\red,\mu}))\cdot\nu \\
	&\quad +2k_\mu(d_{\mu_l}w_{\red,\mu},d_\nu u_{\red,\mu}-d_\nu z_{\red,\mu})-a_\mu(d_{\mu_l}w_{\red,\mu}, d_\nu p_{\red,\mu}) \\
	&= \resd^{\du,d_{\mu_i}}(u_{\red, \mu}, p_{\red, \mu}, d_{\mu_i}u_{\red, \mu}, d_{\mu_i}p_{\red, \mu})[d_{\mu_l}w_{\red,\mu}]  
	- 2k_\mu(d_{\mu_l}w_{\red,\mu},d_\nu z_{\red,\mu}) \\
	&\quad - \partial_\mu (2k_\mu(z_{\red,\mu},d_{\mu_l}w_{\red,\mu})-a_\mu(d_{\mu_l}w_{\red,\mu},w_{\red,\mu}))\cdot\nu \\
	&\leq \Big(\|\resd^{\du,d_{\mu_i}}(u_{\red, \mu}, p_{\red, \mu}, d_{\mu_i}u_{\red, \mu}, d_{\mu_i}p_{\red, \mu}) \| + 2\cont{k_\mu}\|d_{\mu_l}z_{\red,\mu}\| \\
	&\qquad \qquad + 2\cont{\partial_\mu k_\mu}\|z_{\red,\mu}\| + \cont{\partial_\mu \bformd}\|w_{\red,\mu}\| \Big) \|d_{\mu_l}w_{\red,\mu}\|.
	\end{align*}	
\end{proof}

\end{document}